\documentclass[final,leqno,onefignum,onetabnum]{siamltex1213}

\usepackage{xcolor}

\usepackage{amsmath} 
\usepackage{amssymb}

\usepackage{multicol}

\usepackage{enumerate}
\usepackage{xfrac}

\newif\ifAndo
\Andofalse

\newtheorem{example}[theorem]{Example}

\newtheorem{ass}{Assumption}[section]  
\newtheorem{remark}[theorem]{Remark}

\newcommand  \esssup {\mathop{\text{ess} \sup} }

\newcommand \eps   {\varepsilon}

\newcommand \N   {\mathbb{N}}
\newcommand \R   {\mathbb{R}}

\newcommand \K   {\mathcal{K}}
\newcommand \Kinf{\mathcal{K_\infty}}

\newcommand \KL  {\mathcal{KL}}

\newcommand \LL  {\mathcal{L}}

\newcommand{\Uc}{\ensuremath{\mathcal{U}}}

\newcommand \qrq   {\quad\Rightarrow\quad}
\newcommand \srs   {\ \ \Rightarrow\ \ }

\newcommand \Iff   {\Leftrightarrow}

\newcommand \id  {\operatorname{id}}

\newcounter{syscounter}
\newenvironment{sysnum}{\begin{list}{($\Sigma{\arabic{syscounter}}$)}%
{\settowidth{\labelwidth}{($\Sigma4$)}
\settowidth{\leftmargin}{($\Sigma4$)~}%
\usecounter{syscounter}}}
{\end{list}}

\newcommand{\ccat}[3]{{#1\, \underset{#3}{\lozenge}\,{#2}}}

\newcommand{\tm}{\times}
\newcommand{\midset}{\;:\;}

\usepackage{csquotes}  
\newcommand\q{\enquote}

\usepackage{xifthen}

\newcommand{\einsnorm}[2]{\ensuremath{
    \!\!\;\!\!\!\;
    \left\bracevert\!\!\!\!\!\left\bracevert
    \!
		\ifthenelse{\isempty{#2}}{#1}{#1(#2)}
    \!
      \right\bracevert\!\!\!\!\!\right\bracevert
    \!\!\;\!\!\!\;
  }}
\newcommand{\supnorm}[2]{\ensuremath{
    \!\!\;\!\!\!\;
    \left\bracevert\!\!\!\!\left\bracevert
        {#1}_{#2}
      \right\bracevert\!\!\!\!\right\bracevert
    \!\!\;\!\!\!\;
  }}

\title{
Small gain theorems for general networks of heterogeneous infinite-dimensional systems
}

\author{Andrii Mironchenko \thanks{Andrii Mironchenko is with Faculty of Computer Science and Mathematics, University of
  Passau, Germany (\email{andrii.mironchenko@uni-passau.de}).\quad
	This work is supported by German Research Foundation (DFG), grant: MI 1886/2-1.
	\quad
	Shortened version of the paper was presented at NOLCOS 2019:
A. Mironchenko. Small gain theorems for networks of heterogeneous systems. Proc. of the 11th IFAC Symposium on Nonlinear Control Systems (NOLCOS 2019), pages 925-930, 2019. 
	}
}

\date{\today}

\begin{document}
\maketitle

\begin{abstract}
We prove a small-gain theorem for interconnections of $n$ nonlinear 
heterogeneous input-to-state stable (ISS) control systems of a general nature, covering partial, delay and ordinary differential equations.
Furthermore, for the same class of control systems, we derive small-gain theorems for asymptotic gain, 
uniform global stability and weak input-to-state stability properties.
We show that our technique is applicable for different formulations of ISS property (summation, maximum, semimaximum) and discuss tightness of achieved small-gain theorems.
Finally, we introduce variations of uniform asymptotic gain and uniform limit properties, which are particularly useful for small-gain arguments and characterize ISS in terms of these notions.
\end{abstract}



\pagestyle{myheadings}
\thispagestyle{plain}
\markboth{ISS Small-gain theorems for infinite-dimensional systems}{ISS Small-gain theorems for infinite-dimensional systems}

\section{Introduction}
\label{ex:Intro}

The notion of input-to-state stability (ISS), introduced in \cite{Son89} for ordinary differential equations (ODEs), has become a backbone for much of nonlinear control theory, and is currently a well-developed theory with a firm theoretical basis and such powerful tools for ISS analysis, as Lyapunov and small-gain methods. Broad applications of ISS theory include robust stabilization of nonlinear systems \cite{FrK08}, design of nonlinear observers \cite{ArK01}, analysis of large-scale networks \cite{JTP94, DRW07, DRW10}, etc.

The impact of finite-dimensional ISS theory and the need of proper tools for robust stability analysis of distributed parameter systems resulted in generalizations of ISS concepts to broad classes of infinite-dimensional systems, including partial differential equations (PDEs) with distributed and boundary controls, nonlinear evolution equations in Banach spaces with bounded and unbounded input operators, etc.
 \cite{DaM13, JNP18, JLR08, KaK16b, KaK18, KaK19, KaK19b, MaP11, Mir16, MiI15b, MiW18b, TPT18}.

Techniques developed within the infinite-dimensional ISS theory include characterizations 
of ISS and ISS-like properties in terms of weaker 
stability concepts 
\cite{Mir16, MiW18b}, \cite{JNP18, Sch19c}, constructions of ISS Lyapunov functions for PDEs with distributed and boundary controls 
\cite{DaM13, MaP11, MiI15b, PrM12, TPT18, ZhZ18}, non-coercive ISS Lyapunov functions \cite{MiW18b, JMP20}, efficient methods for study of boundary control systems 
\cite{JNP18, JSZ19, KaK16b, KaK19, ZhZ19b}, etc.
For a survey on the infinite-dimensional ISS theory we refer to \cite{MiP20}.

One of the cornerstones of the mathematical control theory is the analysis of interconnected systems.
Large-scale nonlinear systems can be very complex, so that a direct stability analysis of such systems is seldom possible. Small-gain theorems, which are one of the pillars on which the ISS theory stands, make it possible to overcome this obstacle. These theorems guarantee input-to-state stability of an interconnected system, provided all subsystems are ISS and the interconnection structure, described by gains, satisfies the small-gain condition.

\subsection{Existing ISS small-gain results}
There are two kinds of nonlinear small-gain theorems: theorems in terms of trajectories and in terms of Lyapunov functions.
In \emph{small-gain theorems in the trajectory formulation} one assumes that each subsystem is ISS both w.r.t. external inputs and internal inputs from other subsystems, and the so-called internal gains of subsystems characterizing the influence of subsystems on each other are known. The small-gain theorem states that the feedback interconnection is ISS provided the gains satisfy some kind of small-gain conditions.
First small-gain theorems of this type have been developed in \cite{JTP94} for feedback couplings of two ODE systems and in \cite{DRW07} for couplings of $n$ ODE systems.

In \emph{Lyapunov small-gain theorems} it is assumed that all subsystems are ISS w.r.t. external and internal inputs and the ISS Lyapunov functions for subsystems are given together with the corresponding Lyapunov gains. If Lyapunov gains satisfy the small-gain condition, then the whole interconnection is ISS and moreover, an ISS Lyapunov function for the overall system can be constructed.
For couplings of 2 systems such theorems have been shown in \cite{JMW96} and this result has been extended to couplings of $n$ nonlinear ODE systems in \cite{DRW10}. 

As was shown in \cite{DaM13}, ISS small-gain theorems in a Lyapunov formulation can be extended to interconnections of $n$ infinite-dimensional systems without radical changes in the formulation and proof technique. Generalization of integral ISS small-gain theorems is more difficult primarily due to the fact that the state spaces for subsystems should be carefully chosen (see \cite{MiI15b}), although the formulation of the small-gain theorem itself remains similar to the ODE case.
ISS small-gain theorems in terms of vector Lyapunov functions have been reported in \cite{KaJ07,KaJ11}.

\emph{The case of trajectory-based infinite-dimensional small-gain theorems for couplings of $n>2$ systems is significantly more complicated} since the proof of such theorems for the ODE case (see \cite{DRW07}) is based on the fundamental result that ISS of ODE systems is equivalent to uniform global stability (UGS) combined with the asymptotic gain (AG) property shown in \cite{SoW96, SoW95}. Such a characterization is not valid for infinite-dimensional systems, as argued in \cite{Mir16, MiW18b} which makes the proof of \cite{DRW07} not applicable to infinite-dimensional systems without substantial modifications.

For time-delay systems, which are a special case of infinite-dimensional systems, several small-gain results are available.
To the knowledge of the author, the first small-gain results for delay systems have been obtained in \cite{PTM06}, where small-gain theorems for couplings of $2$ \emph{weakly ISS} (i.e. possessing UGS and AG properties) time-delay systems have been derived, guaranteeing weak ISS property for the interconnection provided that the small-gain condition holds.
Small-gain theorems for \emph{weak ISS} of networks of $n$ time-delay systems have been obtained in \cite{TWJ09} and \cite{PDT13}.

The first ISS small-gain theorems, applicable for time-delay systems have been achieved in \cite{KaJ07}, where the small-gain theorems in terms of trajectories (in maximum formulation) have been shown for couplings of two control systems of a rather general nature, covering in particular time-delay systems. 
In \cite{TWJ12} ISS small-gain theorems (in maximum formulation) for couplings of $n\ge 2$ time-delay systems have been obtained by using a Razumikhin-type argument, motivated by \cite{Tee98}.

Recently in \cite{BLJ18} the small-gain theorems for couplings of $n$ input-to-output stable (IOS) evolution equations in Banach spaces have been derived. As a special case of these results, the authors obtain a small-gain theorem for networks of $n$ ISS systems in the maximum formulation. 
Application of small-gain theorems for stability analysis of coupled parabolic-hyperbolic PDEs has been performed in \cite{KaK18}.
Small-gain based boundary feedback design for global exponential stabilization of 1-D semilinear parabolic PDEs has been proposed in \cite{KaK19b}.

\subsection{Contribution}
\textit{Our main results are ISS small-gain theorems (in summation, semimaximum and maximum formulations) for feedback interconnections of $n$ nonlinear heterogeneous control systems whose components belong to a broad class of control systems covering PDEs, time-delay systems, ODE, etc.
Furthermore, we show small-gain theorems for AG, UGS and weak ISS properties.
}

For the description of the interconnections of control systems, we adopt an approach employed in \cite[Definition 3.3]{KaJ07}.
Furthermore, in this paper, we use a variation of a uniform asymptotic gain (UAG) property, which we call bounded input uniform asymptotic gain (bUAG) property.
Although it was used so far not so often as the standard UAG property (but see, e.g., \cite{Tee98}, \cite[Proposition 1.4.3]{Mir12}),
it is more flexible in use and in most cases as powerful as the standard UAG property. 

The ISS small-gain theorem for the summation formulation of the ISS property (Theorem~\ref{thm:ISS_SGT}) is achieved in 3 steps: 
\begin{itemize}
    \item[(i)] We derive uniform global stability of the interconnection (which is the UGS small-gain Theorem~\ref{thm:UGS_SGT}) using the methods developed in \cite{DRW07}. 
    \item[(ii)] We show that the network satisfies bUAG property (the main technical step).
    \item[(iii)] We show that UGS $\wedge$ bUAG is equivalent to ISS, which concludes the proof. Here we base ourselves on characterizations of ISS obtained in \cite{MiW18b}.
\end{itemize}
Based on this strategy, we prove in Sections~\ref{sec:Semimaximum_formulation}, \ref{sec:Maximum_formulation_new} also the ISS small-gain theorems in semimaximum (Theorem~\ref{thm:ISS_SGT_semimax_form}) and maximum (Theorem~\ref{thm:ISS_SGT_max_form}) formulations.
In Section~\ref{sec:Tightness_small-gain-theorem} we argue that the obtained small-gain theorems are tight.
In Section~\ref{sec:Discussion} we discuss in detail the generality and flexibility of the obtained results.

This paper is motivated by the ISS small-gain theorems for networks of $n\in\N$ ODE systems, reported in \cite{DRW07}, and recovers these results as a special case.
As a particular application of our general small-gain theorems, one can obtain novel small-gain results for couplings of $n$ nonlinear time-delay systems and for evolution equations in Banach spaces. In Section~\ref{sec:Discussion} we argue that even in these special cases our results are new and broadly extend the machinery, available for these system classes.

Motivated by the notion of strong ISS, introduced in \cite{MiW18b} and studied in \cite{MiW18b}, \cite{NaS18}, in \cite{Sch19c,ScZ18} the concept of \emph{weak input-to-state stability (wISS)} has been introduced and investigated, in particular, in the context of robust stabilization of port-Hamiltonian systems, see \cite{ScZ18}. 
\emph{In Section~\ref{sec:wISS_SGT} we derive a small-gain result for a weak ISS property.}

Finally, we note that the \emph{characterizations of ISS of infinite-dimensional systems} in terms of weaker stability properties, derived in \cite{MiW18b}, are a key ingredient in the proof of the small-gain theorems in this paper. These characterizations have been already useful in several further contexts, for instance in the non-coercive Lyapunov function theory, see \cite{MiW18b, JMP20}  and in the characterization of the practical ISS property \cite{Mir19a}.
\emph{We hope that refined and more flexible versions of these characterizations derived in Section~\ref{sec:New_ISS_Characterizations} will find further applications within the infinite-dimensional ISS theory.}

\subsection{Notation}
We denote $\R_+:=[0,\infty)$ and $\R^n_+:= \{x\in\R^n:x\geq 0\}$.
For arbitrary $x,y \in \R^n$ define the relation \q{$\geq$} on $\R^n$ by:
$
x \geq y \ \Iff \   x_i \geq y_i,\ \forall i=1,\ldots,n.
$
By \q{$\not\geq$} we understand the logical negation of \q{$\geq$}, that is $x \not\geq y$ iff $\exists i$: $x_i < y_i$.

By $|\cdot|$ we denote the Euclidean norm in finite-dimensional spaces, and by $\id$ we denote the identity operator.
For a function $v:\R_+\to X$, where $X$ is a certain set, we define its restriction to the
interval $[s_1,s_2]$ by
$v_{[s_1,s_2]}(t):=
\begin{cases}
  v(t) & \mbox{~if~}t\in[s_1,s_2],\\
  0 & \mbox{~else.}
\end{cases} $

By $\left\|v\right\|_{[0,t]}$ we denote the supremum norm of $v_{[0,t]}$, i.e. $\left\|v\right\|_{[0,t]}:=\sup_{s\in[0,t]}\|v(s)\|_X$.

By $\Uc:=PC_b(\R_+,U)$ we denote the space of globally bounded, piecewise continuous functions $u:\R_+ \to U$, which
are right continuous. The norm of $u \in \Uc$ is defined by $\|u\|_{\Uc}:=\sup_{t \geq 0}\|u(t)\|_U$. 
Also we use the following classes of comparison functions:
\begin{equation*}
\begin{array}{ll}
{\K} &:= \left\{\gamma:\R_+ \to \R_+ \ : \ \gamma\mbox{ is continuous and strictly increasing, }\gamma(0)=0\right\}\\
{\K_{\infty}}&:=\left\{\gamma\in\K \ :\ \gamma\mbox{ is unbounded}\right\}\\
{\LL}&:=\left\{\gamma:\R_+ \to \R_+ \ :\ \gamma\mbox{ is continuous and decreasing with}
 \lim\limits_{t\rightarrow\infty}\gamma(t)=0 \right\}\\
{\KL} &:= \left\{\beta: \R_+^2 \to \R_+ \ : \ \beta(\cdot,t)\in{\K},\ \forall t \geq 0,\  \beta(r,\cdot)\in {\LL},\ \forall r >0\right\}
\end{array}
\end{equation*}


\section{Problem formulation}
\label{Problem_Formulation}

\subsection{Definition of a control system}

We define the concept of a (time-invariant control) system in the following way:
\begin{definition}
\label{Steurungssystem}
Consider the triple $\Sigma=(X,\Uc,\phi)$ consisting of 
\begin{enumerate}[(i)]  
    \item A normed vector space $(X,\|\cdot\|_X)$, called the \emph{state space}, endowed with the norm $\|\cdot\|_X$.
    \item A normed vector \emph{space of inputs} $\Uc \subset \{u:\R_+ \to U\}$          
endowed with a norm $\|\cdot\|_{\Uc}$, where $U$ is a normed vector \emph{space of input values}.
We assume that the following two axioms hold:
                    
\emph{The axiom of shift invariance}: for all $u \in \Uc$ and all $\tau\geq0$ the time
shift $u(\cdot + \tau)$ belongs to $\Uc$ with \mbox{$\|u\|_\Uc \geq \|u(\cdot + \tau)\|_\Uc$}.

\emph{The axiom of concatenation}: for all $u_1,u_2 \in \Uc$ and all $t>0$ the concatenation of $u_1$ and $u_2$ at time $t$, defined by
$\ccat{u_1}{u_2}{t}(\tau):=
\begin{cases}
u_1(\tau), & \text{ if } \tau \in [0,t], \\ 
u_2(\tau-t),  & \text{ otherwise},
\end{cases}$
%
belongs to $\Uc$.

    \item A map $\phi:D_{\phi} \to X$, $D_{\phi}\subseteq \R_+ \times X \times \Uc$ (called \emph{transition map}), such that for all $(x,u)\in X \tm \Uc$ it holds that $D_{\phi} \cap \big(\R_+ \times \{(x,u)\}\big) = [0,t_m)\tm \{(x,u)\} \subset D_{\phi}$, for a certain $t_m=t_m(x,u)\in (0,+\infty]$.
        
        The corresponding interval $[0,t_m)$ is called the \emph{maximal domain of definition} of $t\mapsto \phi(t,x,u)$.
        
\end{enumerate}
The triple $\Sigma$ is called a \emph{(control) system}, if the following properties hold:

\begin{sysnum}
    \item\label{axiom:Identity} \emph{The identity property:} for every $(x,u) \in X \times \Uc$
          it holds that $\phi(0, x,u)=x$.
    \item \emph{Causality:} for every $(t,x,u) \in D_\phi$, for every $\tilde{u} \in \Uc$, such that $u(s) =
          \tilde{u}(s)$ for all $s \in [0,t]$ it holds that $[0,t]\tm \{(x,\tilde{u})\} \subset D_\phi$ and $\phi(t,x,u) = \phi(t,x,\tilde{u})$.
    \item \label{axiom:Continuity} \emph{Continuity:} for each $(x,u) \in X \times \Uc$ the map $t \mapsto \phi(t,x,u)$ is continuous on its maximal domain of definition.
        \item \label{axiom:Cocycle} \emph{The cocycle property:} for all
                  $x \in X$, $u \in \Uc$, for all $t,h \geq 0$ so that $[0,t+h]\times \{(x,u)\} \subset D_{\phi}$, we have
$\phi(h,\phi(t,x,u),u(t+\cdot))=\phi(t+h,x,u)$.
\end{sysnum}

\end{definition}

This class of systems encompasses control systems generated by ODEs, switched systems, time-delay systems,
many classes of PDEs, important classes of boundary control systems and many other systems. 


\begin{definition}
\label{def:FC_Property} 
We say that a control system (as introduced in Definition~\ref{Steurungssystem}) is \emph{forward complete (FC)}, if 
$D_\phi = \R_+ \tm X\tm\Uc$, that is for every $(x,u) \in X \times \Uc$ and for all $t \geq 0$ the value $\phi(t,x,u) \in X$ is well-defined.
\end{definition}

\begin{remark}
\label{rem:Systems_classes_here_and_in_MiW18b} 
The class of forward complete control systems considered in this paper is precisely the class of control systems considered in \cite{MiW18b}, where forward completeness axiom was a part of a definition of a control system. Hence the results obtained in \cite{MiW18b} are applicable to forward complete systems considered in this paper.
\end{remark}

Recall an important property of equations in Banach spaces with bounded input operators and Lipschitz continuous right-hand sides (see \cite[Theorem 4.3.4]{CaH98}): 
if the solution stays bounded over $[0,t)$, then it can be prolonged to $[0,t+\varepsilon)$ for a certain $\varepsilon>0$.
The next property, adopted from \cite[Definition 1.4]{KaJ11b} formalizes this behavior for general control systems.
\begin{definition}
\label{def:BIC} 
We say that a system $\Sigma$ satisfies the \emph{boundedness-implies-continuation (BIC) property} if for each
$(x,u)\in X \tm \Uc$ such that the maximal existence time $t_{m}=t_m(x,u)$ is finite, and for all $M > 0$, there exists $t \in [0,t_m)$ with $\|\phi(t,x,u)\|_X>M$.
\end{definition}

\subsection{Interconnections of control systems}
\label{sec:Interconnections_definition}

Let $(X_i,\|\cdot\|_{X_i})$, $i=1,\ldots,n$ be normed vector spaces endowed with the corresponding norms.
Define for each  $i=1,\ldots,n$ the normed vector space
\vspace{-2mm}
\begin{eqnarray*}
 X_{\neq i}:=X_1 \times \ldots \times X_{i-1} \times X_{i+1} \times \ldots \times X_n, \qquad 
\|x\|_{X_{\neq i}} : = \Big(\textstyle\sum_{j=1,\ j\neq i}^n\|x_j\|^2_{X_j}\Big)^{\sfrac{1}{2}}.
\end{eqnarray*}
Let control systems $\Sigma_i:=(X_i,PC_b(\R_+,X_{\neq i})\tm \Uc,\bar{\phi}_i)$ be given. 
We call $X_{\neq i}$ the \emph{space of internal input values} and $PC_b(\R_+,X_{\neq i})$ the \emph{space of internal inputs}.
The norm on $PC_b(\R_+,X_{\neq i})\tm \Uc$ we define as
\vspace{-2mm}
\begin{eqnarray}
\|(v,u)\|_{PC_b(\R_+,X_{\neq i})\tm \Uc} := \Big(\textstyle\sum_{j \neq i}\|v_j\|^2_{PC_b(\R_+,X_{j})} + \|u\|^2_{\Uc}\Big)^{\sfrac{1}{2}}.
\label{eq:Norm_Full_Input}
\end{eqnarray}
Define also the normed vector space
\vspace{-2mm}
\begin{eqnarray}
X=X_{1}\times\ldots\times X_{n}, \qquad \|x\|_X := \Big(\textstyle\sum_{i=1}^n\|x_i\|^2_{X_i}\Big)^{\sfrac{1}{2}},
\label{eq:State_Space_Full_System}
\end{eqnarray}
and assume that there is a map $\phi=(\phi_1,\ldots,\phi_n):D_\phi \to X$, defined over a certain domain $D_{\phi} \subseteq \R_+ \tm X \times \Uc$
so that for each $x=(x_1,x_2,\ldots,x_n) \in X$, each $u\in\Uc$ and all $t\in\R_+$ so that $(t,x,u)\in D_{\phi}$
and for every $i=1,\ldots,n$, it holds that 
\begin{eqnarray}
\phi_i(t,x,u) = \bar{\phi}_i\big(t,x_i,(v_i,u)\big),
\label{eq:Sigma_i_models_i_th_mode_of_Sigma}
\end{eqnarray}
where\qquad
$v_i(t) = (\phi_1(t,x,u),\ldots,\phi_{i-1}(t,x,u),\phi_{i+1}(t,x,u),\ldots,\phi_{n}(t,x,u)).
$

Assume further that $\Sigma:=(X,\Uc,\phi)$ is a control system with the state space $X$, input space $\Uc$ and with a BIC property.
Then $\Sigma$ is called \emph{a (feedback) interconnection} of systems $\Sigma_1,\ldots,\Sigma_n$.

In other words, condition \eqref{eq:Sigma_i_models_i_th_mode_of_Sigma} means that if the modes $\phi_j(\cdot,x,u)$, $j\neq i$ of the system $\Sigma$ will be sent to $\Sigma_i$ as the internal inputs (together with an external input $u$), and the initial state will be chosen as $x_i$ (the $i$-th mode of $x$), then the resulting trajectory of the system $\Sigma_i$, which is $\bar{\phi}_i(\cdot,x_i,v,u)$ will coincide with the trajectory of the $i$-th mode of the system $\Sigma$ on the interval of existence of $\phi_i$.
\begin{remark}
\label{rem:Unboundedness-of-internal-inputs} 
Note that $\phi_j(\cdot,x,u)$ may be unbounded for any $j$, and thus $(v_i,u)$ does not have to be an admissible input to the system $\Sigma_i$. However, as $v_i$ is bounded on compact time intervals, and since $\Uc$ satisfies the concatenation property, for any $t>0$ it holds that $(\ccat{v_i}{0}{t},\ccat{u}{0}{t}) \in PC_{b}(\R_+,X_{\neq i}) \tm \Uc$, and we can extend the control system $\Sigma_i$ to a larger class of inputs of $\Sigma_i$ by defining 
$\bar{\phi}_i\big(t,x_i,(v_i,u)\big):=\bar{\phi}_i\big(t,x_i,(\ccat{v_i}{0}{t},\ccat{u}{0}{t})\big)$.
The obtained control system on a larger space of inputs could be called a causal extension of $\Sigma_i$ and will be denoted again by $\Sigma_i$. We understand $\Sigma_i$ in the sense of such a causal extension.
\end{remark}

Note that the trajectory of each $\Sigma_i$ depends continuously on time due to the continuity axiom. However, as the space of continuous functions does not satisfy the concatenation property, we enlarge it to include the piecewise continuous functions. This motivates the choice of the space $PC_b(\R_+,X_{\neq i})\tm \Uc$ as the input space for the $i$-th subsystem.

\begin{remark}
\label{rem:Interconnection Philosophy} 
This definition of feedback interconnections, which we adopted from \cite[Definition 3.3]{KaJ07}, does not depend on a particular type of control systems which are coupled, and applies to large-scale systems, consisting of heterogeneous components as PDEs, time-delay systems, ODE systems, etc.
The definition also applies to both in-domain and boundary interconnections of PDE systems.
\hfill$\square$
\end{remark}

Next, we show how the couplings of ODEs and of evolution equations in Banach spaces can be represented in our approach.
Many further classes of systems can be treated in a similar way.

\subsection{Example: Interconnections of ODE systems}
\label{sec:Example: interconnections of ODE systems}

Consider the interconnected systems of finitely many ODEs
\begin{equation}
\label{eq:Coupled_n_ODEs}
\dot{x}_{i} = f_{i}(x_{1},\ldots,x_{n},u), \qquad i=1,\ldots,n,
\end{equation}
Here we assume that the state space of the $i$-th subsystem is $X_i:=\R^{p_i}$ endowed with the Euclidean norm. 
Define $N:=p_1 +\ldots + p_n$. 
Then the state space of the whole system, defined by \eqref{eq:State_Space_Full_System} is $X:=\R^N$, endowed with the Euclidean norm.
We assume that inputs belong to the space $\Uc:=L_{\infty}(\R_+,\R^{m})$.

Defining for $x_i \in \R^{p_i},\ i=1,\ldots, n$ the state 
$x=(x_{1},\ldots,x_{n})^T$, and $f(x,u)=(f_{1}(x,u),\ldots,f_{n}(x,u))^T$, we can rewrite the coupled system in the form
\begin{equation}
\label{eq:ODE_coupled_sys}
\dot{x}(t)=f(x(t),u(t)), \quad u(t) \in U.
\end{equation}
Assuming that $f$ is Lipschitz continuous w.r.t. the state, 
for each initial condition and for each external input $u\in\Uc$ there is a unique absolutely continuous (mild) solution of \eqref{eq:Coupled_n_ODEs} and \eqref{eq:Coupled_n_ODEs} is a well-defined control system, which is an interconnection of the systems $\Sigma_i$, $i=1,\ldots,n$.

\subsection{Example: Interconnections of retarded systems}
\label{sec:Example: interconnections of retarded systems}

Consider a coupled system of retarded differential equations of the form
\begin{equation}
\label{eq:time-delay-network}
\dot{x}_i(t)=f_i(x_{1,t}, \ldots,x_{n,t},u_t), \quad i=1,\ldots,n,
\end{equation}
where we denote $x_{i,t}(s):=x_i(t+s)$, $u_t(s)=u(t+s)$, $s\in[-T_d,0]$, for all $t\geq 0$,
and $T_d>0$ is the fixed (maximal) time-delay.

The state space for the $i$-th subsystem we choose as $X_i:= C([-T_d,0],\R^{p_i})$, endowed with the supremum norm, defined for any $x\in X_i$ by $\|x\|_{X_i}:=\sup_{t\in[-T_d,0]}|x_i(t)|.$

We assume here that $U:=\R^m$ and that input $u$ belongs to the space \linebreak
$\Uc:=L_\infty([-T_d,+\infty),U)$ of globally essentially bounded, measurable functions \linebreak ${u:[-T_d,+\infty) \to U}$. The norm of $u \in \Uc$ is given by
$\|u\|_{\Uc}:=\esssup_{t \geq -T_d}|u(t)|$. 

Define $N:=p_1 +\ldots + p_n$. 
The state space of the network \eqref{eq:time-delay-network} can be chosen as $X:= C([-T_d,0],\R^{N})$, again endowed with the supremum norm.

Defining for $x_i \in X_i,\ i=1,\ldots, n$ the state 
$x=(x_{1},\ldots,x_{n})^T$, $f(x_t,u_t)=(f_{1}(x_t,u_t),\ldots,f_{n}(x_t,u_t))^T$, we can rewrite the coupled system in the form
\begin{equation}
\label{eq:time-delay} 
\dot{x}(t)=f(x_t,u_t).
\end{equation}

We say that $\zeta \in C([-T_d,\tau],\R^N)$, $\tau>0$ is a \emph{solution of \eqref{eq:time-delay} on $[-T_d,\tau)$} subject to an initial condition $x \in X$ and an input $u\in\Uc$, if $\zeta$ is absolutely continuous on $[-T_d,\tau]$, satisfies the initial condition $\zeta(s) = x(s)$ for $s\in[-T_d,0]$ and the equation $\dot{\zeta}(t) = f(\zeta_t,u_t)$ holds almost everywhere on $(0,\tau)$.

With this solution concept, and under natural assumptions on the right-hand side $f$ one can show that the interconnection is well-posed and satisfies the BIC property.
We refer to \cite[Section 7]{MiP20} for details and further references.

\subsection{Example: Interconnections of evolution equations in Banach spaces}
\label{sec:Example: interconnections of evolution equations in Banach spaces}

Consider a system of the following form 
\begin{equation}
\label{Kopplung_N_Systeme}
\dot{x}_{i}=A_{i}x_{i}+f_{i}(x_{1},\ldots,x_{n},u),\qquad i=1,\ldots,n,
\end{equation}
where the state space of the $i$-th subsystem $X_{i}$ is a Banach space and $A_i$ with the domain of definition $D(A_i)$ is the generator of a
$C_{0}$-semigroup on $X_{i}$, $i=1,\ldots, n$.
In the sequel, we assume that the set of input values $U$ is a normed
linear space and that the input functions belong to space
$\Uc:=PC_b(\R_+,U)$.

Define the state space $X$ of the whole system \eqref{Kopplung_N_Systeme} by \eqref{eq:State_Space_Full_System}.
We choose further the input space to the $i$-th subsystem as \eqref{eq:Norm_Full_Input}.

For $x_i \in X_i,\ i=1,\ldots, n$ define
$x=(x_{1},\ldots,x_{n})^T$, $f(x,u)=(f_{1}(x,u),\ldots,f_{n}(x,u))^T$.
By $A$ we denote the diagonal operator $A:=\diag(A_{1},\ldots,A_{n})$.
 Domain of definition of $A$ is given by $D(A)=D(A_{1})\times\ldots\times D(A_{n})$.
It is well-known that $A$ is the generator of a $C_{0}$-semigroup $\{T(t):t\geq 0\}$  of bounded operators on $X$.

With this notation the coupled system \eqref{Kopplung_N_Systeme} takes the form
\begin{equation}
\label{InfiniteDim}
\dot{x}(t)=Ax(t)+f(x(t),u(t)), \quad u(t) \in U.
\end{equation}
Assuming that $f$ is Lipschitz continuous w.r.t. $x$ guarantees that the mild solutions of \eqref{InfiniteDim} exists and is unique for every initial condition and for any admissible input. Here mild solutions $x:[0,\tau] \to X$ are the solutions of the integral equation
\begin{align}
\label{InfiniteDim_Integral_Form}
x(t)=T(t)x(0) + \int_0^t T(t-s)f(x(s),u(s))ds,
\end{align}
belonging to the space of continuous functions $C([0,\tau],X)$ for some $\tau>0$.

Under these assumptions the system \eqref{InfiniteDim} can be seen as a well-defined interconnection of the systems $\Sigma_i$, $i=1,\ldots,n$, and each $\Sigma_i$ is a well-defined system in the sense of Definition~\ref{Steurungssystem}.
Moreover, by a variation of \cite[Theorem 4.3.4]{CaH98} one can show that \eqref{InfiniteDim_Integral_Form} possesses the BIC property.

\section{Stability notions}
\label{sec:Stability_notions}

The main concept in this paper is:
\begin{definition}
\label{Def:ISS}
A system $\Sigma=(X,\Uc,\phi)$, with $\phi:D_\phi\to X$ is called \emph{(uniformly)  input-to-state stable
(ISS)}, if there exist $\beta \in \KL$ and $\gamma \in \Kinf\cup\{0\}$ 
so that 
\begin {equation}
\label{iss_sum}
(t,x,u) \in D_{\phi} \qrq \| \phi(t,x,u) \|_{X} \leq \beta(\| x \|_{X},t) + \gamma( \|u\|_{\Uc}).
\end{equation}
\end{definition}
Another important property is  
\begin{definition}
\label{Def:0-UGAS}
A system $\Sigma=(X,\Uc,\phi)$, with $\phi:D_\phi\to X$ is called \emph{ uniformly globally asymptotically stable for a zero input (0-UGAS)}, if there is $\beta \in \KL$ so that 
\begin {equation}
\label{eq:UGAS}
(t,x,0) \in D_{\phi}  \qrq \| \phi(t,x,0) \|_{X} \leq \beta(\| x \|_{X},t).
\end{equation}
\end{definition}
Clearly, ISS implies 0-UGAS.
Another important property implied by ISS is 
\begin{definition}
A system $\Sigma=(X,\Uc,\phi)$, with $\phi:D_\phi\to X$ is called
 \emph{ uniformly globally stable (UGS)}, if there exist $ \sigma \in\Kinf$ and $\gamma
          \in \Kinf \cup \{0\}$ such that 
\begin{equation}
\label{GSAbschaetzung}
(t,x,u) \in D_{\phi} \qrq \left\| \phi(t,x,u) \right\|_X \leq \sigma(\|x\|_X) + \gamma(\|u\|_{\Uc}).
\end{equation}
\end{definition}
For a normed linear space $W$ and any $r>0$ denote $B_{r,W} :=\{u \in W: \|u\|_W < r\}$. 
If $W$ is a state space $X$ we write simply $B_r$ instead of $B_{r,X}$.

A local counterpart of the UGS property is
\begin{definition}
A system $\Sigma=(X,\Uc,\phi)$, with $\phi:D_\phi\to X$ is called
 \emph{ uniformly locally stable (ULS)}, if there exist $ \sigma \in\Kinf$, $\gamma
          \in \Kinf \cup \{0\}$ and $r>0$ such that for all $x \in B_r$, $u\in \overline{B_{r,\Uc}}$ and all $t\geq 0$ so that $(t,x,u) \in D_{\phi}$, the estimate \eqref{GSAbschaetzung} holds.
\end{definition}

\begin{lemma}
\label{lem:ISS_UGS_and_BIC_property} 
Let $\Sigma=(X,\Uc,\phi)$ be a UGS control system. If $\Sigma$ has the BIC property, then $\Sigma$ is forward complete.
\end{lemma}

\begin{proof}
Pick any $x \in X$ and $u\in\Uc$. Then there is a maximal existence time $t_m$ so that $(t,x,u)\in D_\phi$ for all $t \in [0, t^*)$.
Assume that $t_m < \infty$. As $\Sigma$ is UGS, $\limsup_{t \to t_m-0}\|\phi (t, x, u)\|_X<\infty$, and we obtain a contradiction to the BIC property of $\Sigma$. Hence $t_m=+\infty$ and $\Sigma$ is forward complete.
\end{proof}

For forward complete systems we introduce the following asymptotic properties
\begin{definition}
A forward complete system $\Sigma=(X,\Uc,\phi)$ has the
\begin{itemize}
    \item[(i)] \emph{ asymptotic gain (AG) property}, if there is $
          \gamma \in \Kinf  \cup \{0\}$ such that for all $\eps >0$, 
          all $x \in X$ and for all $u \in \Uc$ there is a
          $\tau=\tau(\eps,x,u) < \infty$ such that 
\begin{equation}
\label{AG_Absch}
t \geq \tau\ \quad \Rightarrow \quad \|\phi(t,x,u)\|_X \leq \eps + \gamma(\|u\|_{\Uc}).
\end{equation}

  \item[(ii)] \emph{ strong asymptotic gain (sAG) property}, if there is $
    \gamma \in \Kinf \cup \{0\}$ such that for all $x \in X$ and all
    $\eps >0$ there is a $\tau=\tau(\eps,x) < \infty $ such that 
for all $u \in \Uc$ the estimate \eqref{AG_Absch} holds.
%

\item[(iii)] \emph{ bounded input uniform asymptotic gain (bUAG) property}, if there
          exists $ \gamma \in \Kinf \cup \{0\}$ such that for all $ \eps, r
          >0$ there is a $ \tau=\tau(\eps,r) < \infty$ such
          that for all $u \in \Uc$: $\|u\|_{\Uc}\leq r$ and all $x \in B_{r}$
					 the estimate \eqref{AG_Absch} holds.

    \item[(iv)] \emph{ uniform asymptotic gain (UAG) property}, if there
          exists  $ \gamma \in \Kinf \cup \{0\}$ such that for all $ \eps, r
          >0$ there is a $ \tau=\tau(\eps,r) < \infty$ such
          that for all $u \in \Uc$ and all $x \in B_{r}$ the implication \eqref{AG_Absch} holds.

\end{itemize}
\end{definition}

All types of asymptotic gain properties imply that all trajectories converge to the ball of radius $\gamma(\|u\|_{\Uc})$ around the origin as $t \to \infty$. The difference between AG, sAG, bUAG, and UAG is in the kind of dependence of
$\tau$ on the states and inputs. 
In UAG systems this time depends (besides $\eps$) only on the norm of the state, in sAG systems, it depends on the state $x$ (and may vary for different states with the same norm), but it does not depend on $u$. In AG systems $\tau$ depends both on $x$ and on $u$.

The following lemma shows how bUAG property can be \q{upgraded} to the UAG and ISS properties. This result we will need for development of the small-gain theorems. Later in Theorem~\ref{thm:New_ISS_Characterizations} we show a much stronger characterization of ISS.
\begin{lemma}
\label{lem:UGS_und_bUAG_imply_UAG}
Let $\Sigma=(X,\Uc,\phi)$ be a control system with a BIC property. If $\Sigma$ is UGS and bUAG, then $\Sigma$ is forward compete, UAG and ISS.
\end{lemma}

\begin{proof}
As $\Sigma$ satisfies BIC property and is UGS, then it is forward complete by Lemma~\ref{lem:ISS_UGS_and_BIC_property}  (in particular, the property bUAG assumed for $\Sigma$ makes sense).
 
Pick arbitrary $\eps,r>0$. Let $\tau$ and $\gamma$ be as in the formulation of the bUAG property and pick any $x \in B_r$ and $u\in\Uc$. 
If $\|u\|_\Uc\leq r$, then \eqref{AG_Absch} is the desired estimate.

Let $\|u\|_\Uc > r$. As $\Sigma$ is UGS, we have $\|\phi(t,x,u)\|_X \leq \sigma(\|x\|_X) + \gamma(\|u\|_{\Uc})$ for all $t\geq 0$. 
Here we assume that $\gamma$ is same as in the definition of a bUAG property (otherwise pick the maximum of both).

As $\|u\|_\Uc > r\geq \|x\|_X$, we obtain $\|\phi(t,x,u)\|_X \leq \sigma(\|u\|_\Uc) + \gamma(\|u\|_{\Uc})$,
and thus
\begin{align*}    
x \in B_r \ \wedge \ u\in \Uc \ \wedge \ t \geq \tau \qrq \|\phi(t,x,u)\|_X \leq \eps + \gamma(\|u\|_{\Uc}) +\sigma(\|u\|_\Uc),
\end{align*}
which shows UAG property with the asymptotic gain $\gamma + \sigma$.

As $\Sigma$ is forward complete, UAG and UGS, the ISS property of $\Sigma$ follows by \cite[Theorem 5]{MiW18b}.
\end{proof}

\section{Coupled systems and gain operators}
\label{sec:Gain_operators_SGC}

In this subsection we consider $n$ systems $\Sigma_i:=(X_i,PC_b(\R_+,X_{\neq i}) \tm \Uc,\bar{\phi}_i)$, $i=1,\ldots,n$, where all $X_i$, $i=1,\ldots,n$ and $\Uc$ are normed  linear spaces. Furthermore, we assume that all $\Sigma_i$, $i=1,\ldots,n$ are forward complete.

Stability properties introduced in Section~\ref{sec:Stability_notions} are defined in terms of the norms of the whole input, and this is not suitable for the consideration of coupled systems, as we are interested not only in the collective influence of all inputs over a subsystem but in the influence of particular subsystems over a given subsystem.

Therefore we reformulate the ISS property for a subsystem in the following form:
\begin{lemma}
\label{lem:ISS_reformulation_n_systems}
A forward complete system \emph{$\Sigma_i$ is ISS (in summation formulation)} if and only if there exist $\gamma_{ij}, \gamma_i \in \Kinf\cup\{0\},\ j=1,\ldots,n$ and $\beta_i \in \KL$, such that for all initial values $x_i \in X_i$, all internal inputs 
$w_{\neq i} := (w_1,\ldots,w_{i-1}, w_{i+1},\ldots,w_n) \in PC_b(\R_+,X_{\neq i})$, all external inputs $u \in\Uc$
and all $t \in\R_+$ the following estimate holds:
\begin{eqnarray}
\label{eq:ISS_n_sys_sum}
\|\bar{\phi}_i\big(t,x_i,(w_{\neq i}, u)\big)\|_{X_i}  \leq
 \beta_i\left(\left\|x_i\right\|_{X_i},t\right) + \textstyle\sum_{j\neq i}\gamma_{ij}\big(\left\|w_j\right\|_{[0,t]}\big) + \gamma_i\left(\left\|u\right\|_{\Uc}\right).
\end{eqnarray}
\end{lemma}
The proof of this lemma is similar to the proof of \cite[Lemma 2.4.1]{Mir12}, and is omitted.

The functions $\gamma_{ij}$ and $\gamma_i$ in the statement of Lemma~\ref{lem:ISS_reformulation_n_systems} are called \emph{(nonlinear) gains}.
For notational simplicity we allow the case $\gamma_{ij}\equiv 0$ and
require $\gamma_{ii}\equiv 0$ for all $i$.

Analogously, one can restate the definitions of UGS, AG and bUAG properties:
\begin{lemma}
    \label{lem:UGS_reformulation_n_systems}
$\Sigma_i$  is UGS (in summation formulation) if and only if there exist $\gamma_{ij},\ \gamma_i \in \Kinf\cup\{0\}$ and $\sigma_i \in \Kinf$, such that for all initial values $x_i\in X_i$, all internal inputs $w_{\neq i} := (w_1,\ldots,w_{i-1}, w_{i+1},\ldots,w_n) \in PC_b(\R_+,X_{\neq i})$, all external inputs $u \in\Uc$ and all $t \in\R_+$ the following inequality holds
\begin{eqnarray}
\label{eq:UGS_n_sys_sum}
\|\bar{\phi}_i\big(t,x_i,(w_{\neq i}, u)\big)\|_{X_i}  \leq
 \sigma_i\left(\left\|x_i\right\|_{X_i}\right) + \textstyle\sum_{j\neq i}\gamma_{ij}\big(\left\|w_j\right\|_{[0,t]}\big) + \gamma_i\left(\left\|u\right\|_{\Uc}\right).
\end{eqnarray}
\end{lemma}
\vspace{-4mm}
\begin{lemma}
\label{lem:AG_reformulation_n_systems}
$\Sigma_i$ is AG if and only if there exist $\gamma_{ij},\ \gamma_i \in \Kinf\cup\{0\}$, such that for all $x_i \in X_i$, $u\in\Uc$, $w_{\neq i} := (w_1,\ldots,w_{i-1}, w_{i+1},\ldots,w_n) \in PC_b(\R_+,X_{\neq i})$, $\eps>0$ there is a time $ \tau_i:=\tau_i(x_i,u,w_{\neq i},\varepsilon) < \infty$ so that it holds that
\begin{equation}    
\label{UAG_Absch_ith_subsystem}
t \geq \tau_i \quad \Rightarrow \quad \|\bar{\phi}_i\big(t,x_i,(w_{\neq i}, u)\big)\|_{X_i} \leq \eps + \textstyle\sum_{j\neq i}\gamma_{ij}\left(\left\|w_j\right\|_{\infty}\right) + \gamma_i(\|u\|_{\Uc}).
\end{equation}
\end{lemma}
\vspace{-4mm}
\begin{lemma}
\label{lem:bUAG_reformulation_n_systems}
$\Sigma_i$ is bUAG if and only if there exist $\gamma_{ij},\ \gamma_i \in \K\cup\{0\}$, such that for all $\eps>0$, for all $r>0$ there is $\tau_i=\tau_i(\eps,r) < \infty$ such that 
for all $u \in \Uc$: $\|u\|_{\Uc}\leq r$ and all $x_i \in B_{r}(0,X_i)$, for all 
$w_{\neq i} := (w_1,\ldots,w_{i-1}, w_{i+1},\ldots,w_n) \in PC_b(\R_+,X_{\neq i})$: $\|w_j\|_{\infty}\leq r$ the implication \eqref{UAG_Absch_ith_subsystem} holds.
\end{lemma}

In the above definitions $w_j$, $j=1,\ldots,n$ are general inputs, which are not necessarily related to the states of other subsystems (i.e. we have considered all $\Sigma_i$ as disconnected systems).
Now assume that
\begin{ass}
\label{ass:Interconnection-is-well-posed}
The interconnection of forward complete systems $\Sigma_1,\ldots,\Sigma_n$ as introduced in Section~\ref{sec:Interconnections_definition}, which we call $\Sigma:=(X,\Uc,\phi)$, is a well-defined control system with a BIC property.
\end{ass}

Pick arbitrary $x \in X$ and $u\in\Uc$ and define for $i=1,\ldots,n$ the quantities
\begin{eqnarray}
\phi_i:=\phi_i(\cdot,x,u),\quad \phi_{\neq i} = (\phi_1,\ldots,\phi_{i-1}, \phi_{i+1},\ldots,\phi_n).
\label{eq:phi_neq_i}
\end{eqnarray}
We rewrite the definitions of ISS and UGS, specialized for the inputs $w_j:=\phi_j$, in a shorter vectorized form, using the shorthand notation from \cite{DRW07}, introduced next.

For vector functions $w=(w_1,\ldots,w_n)^T:\R_+\to X_1\times \ldots \times X_n$ such that\linebreak[3]
$w_i \in C(\R_+, X_i), i=1,\ldots,n$ and times $0\leq t_1 \leq t_2$ we
define
\[
\supnorm{w}{[t_1,t_2]}
:= \big( \|w_{i,[t_1,t_2]}\|_\infty\big)_{i=1}^n
	\in\R^n_+.
\]
Furthermore, we introduce for all $t,u$ and all $x=(x_i)_{i=1}^n \in X$ the following notation:
  \begin{gather}
        \label{eq:vector_gamma}
    \einsnorm{\bar{\phi}}{t,x,u}
    := \big(\|\bar{\phi}_i(t,x_i,(\phi_{\neq i}, u))\|_{X_i}\big)_{i=1}^n
		,
    \ \
   \vec{\gamma}(\|u\|_\Uc)
    :=  \big(\gamma_i(\|u\|_\Uc)\big)_{i=1}^n
		,\\
    \mbox{and} \quad        
   \vec{\sigma}(s)
    :=  \big(\sigma_i(s_i)\big)_{i=1}^n
		,
\ \        
    \vec{\beta}(s,t)
    :=    \big(\beta_i(s_i,t)\big)_{i=1}^n
		,
        \  \forall s =(s_i)_{i=1}^n \in\R^n_+,\ t\geq 0.
  \end{gather}

We are going to collect all the internal gains in the matrix $\Gamma:=(\gamma_{ij})_{i,j=1,\ldots,n}$, which we call the \emph{gain matrix}. If the gains are taken from the ISS restatement \eqref{eq:ISS_n_sys_sum}, then we call the corresponding gain matrix $\Gamma^{ISS}$. Analogously, the gain matrices $\Gamma^{UGS}$, $\Gamma^{AG}$, $\Gamma^{UAG}$ are defined.

Now for a given gain matrix $\Gamma$ define the operator $\Gamma_{\boxplus}:\R^n_+\to\R^n_+$ by
\begin{equation}
\Gamma_{\boxplus}(s) := \Big(
  \textstyle\sum_{j=1}^n \gamma_{1j}(s_j),\ldots,\textstyle\sum_{j=1}^n \gamma_{nj}(s_j)
\Big)^T,   \quad s=(s_1,\ldots,s_n)^T\in\R^n_+.
\label{eq:ps_gamma}
\end{equation}

Again, to emphasize that the gains are from the ISS restatement \eqref{eq:ISS_n_sys_sum},
the corresponding gain operator will be denoted by $\Gamma_\boxplus^{ISS}$.

Note that by the
properties of $\gamma_{ij}$ for $s_1,s_2\in\R_+^n$ we have the implication
\begin{equation}\label{eq:monot_Gamma_plus}
s_1\ge s_2\qrq\Gamma_\boxplus(s_1)\ge\Gamma_\boxplus(s_2),
\end{equation}
so that $\Gamma_\boxplus$ defines a monotone (w.r.t. the partial order $\ge$ in $\R^n$) map.

The ISS conditions \eqref{eq:ISS_n_sys_sum} with this notation imply that for $t\geq 0$ it holds that
  \begin{equation}
    \label{eq:ISS subsysteme}
       \einsnorm{\bar{\phi}}{t,x,u} \leq \vec{\beta}(\einsnorm{\bar{\phi}}{0,x,u} ,t) + \Gamma^{ISS}_\boxplus(
    \supnorm{\phi}{[0,t]})
    + \vec{\gamma}(\|u\|_\Uc).
  \end{equation}
Analogously, the UGS conditions \eqref{eq:UGS_n_sys_sum} imply that for $t\geq 0$ it holds that    
  \begin{equation}
    \label{eq:UGS subsysteme}
    \einsnorm{\bar{\phi}}{t,x,u}   \leq \vec{\sigma}(\einsnorm{\bar{\phi}}{0,x,u} ) + \Gamma^{UGS}_\boxplus(\supnorm{\phi}{[0,t]})
    + \vec{\gamma}(\|u\|_\Uc).
  \end{equation}
    
To guarantee the stability of the interconnection $\Sigma$, the following properties of the gain operators $\Gamma^{UGS}_\boxplus$ and $\Gamma^{ISS}_\boxplus$ will be crucial.

 

\begin{definition}
\label{def:SGC} 
We say that a nonlinear operator $A:\R^n_+ \to \R^n_+$ satisfies
\begin{itemize}
    \item the \emph{small-gain condition}, if 
  \begin{equation}
    A(s)\not\geq s,\qquad\forall s\in\R^n_+\setminus\{0\}.
  \label{eq:SGC}
  \end{equation}
    \item the \emph{strong small-gain condition}, if there is $\rho \in \Kinf$ and a corresponding map $D:\R_+^n\to\R_+^n$, defined by $D((s_i)_{i=1}^n) := ((\id+\rho)(s_i))_{i=1}^n$, so that
  \begin{equation}
    (A\circ D)(s)\not\geq s,\qquad\forall s\in\R^n_+\setminus\{0\}.    
  \label{eq:Strong_SGC}
  \end{equation}
\end{itemize}
\end{definition}

We will need the following technical result, see \cite[Theorem 6.1]{Rue10}.
\begin{lemma}
\label{lem:boundedness_ID-Gamma_general case}
If $A=\Gamma_{\boxplus}$ defined by \eqref{eq:ps_gamma}, or $A=\Gamma_{\otimes}$ defined by \eqref{eq:Gamma-max}
 satisfies the strong small gain condition \eqref{eq:Strong_SGC}, then there exists a $\xi\in \Kinf$ such that for all $w,v \in \R^n_+$ it holds that
  \begin{equation}
    \label{eq:ineqbound}
    (\id-A)(w)\leq v \qrq |w|\leq \xi(|v|).
  \end{equation}
\end{lemma}

%
%
%
%

\section{Small-gain theorems for control systems}

In this section, we show small-gain theorems for UGS, ISS, AG and weak ISS properties.

\subsection{UGS small-gain theorem}
\label{sec:UGS_SGT}

We start with a small-gain theorem that guarantees that the coupling of UGS systems is a UGS system provided the strong small-gain condition \eqref{eq:Strong_SGC} holds.
This will in particular show that the coupled system is forward complete.    This result and its proof are an infinite-dimensional 
version of \cite[Theorem 8]{DRW07}.
\begin{theorem}[UGS Small-gain theorem]
\label{thm:UGS_SGT} 
Let $\Sigma_i:=(X_i,PC_b(\R_+,X_{\neq i}) \tm \Uc,\bar{\phi}_i)$, $i=1,\ldots,n$ be control systems, where all $X_i$, $i=1,\ldots,n$ and $\Uc$ are normed  linear spaces.
Assume that $\Sigma_i$, $i=1,\ldots,n$ are forward complete systems, satisfying the UGS estimates as in Lemma~\ref{lem:UGS_reformulation_n_systems}, and that the interconnection $\Sigma=(X,\Uc,\phi)$ is well-defined and possesses the BIC property.

If $\Gamma^{UGS}_\boxplus$ satisfies the strong small gain condition \eqref{eq:Strong_SGC}, then $\Sigma$ is forward complete and UGS.
\end{theorem}

\begin{proof}
Pick any $u\in\Uc$ and any initial condition $x\in X$. As we assume that the interconnection $\Sigma=(X,\Uc,\phi)$ is well-defined, by definition of a  control system, the solution of $\Sigma$ exists on a certain maximal existence interval $[0,t_m)$, where $t \in (0,+\infty]$.
Define $\phi_i$ and $\phi_{\neq i}$ as in \eqref{eq:phi_neq_i}.

According to the definition of the interconnection, for all $i=1,\ldots,n$ it holds that $\phi_i(s,x,u) = \bar{\phi}_i(s,x_i,(\phi_{\neq i}, u))$, $s \in [0,t]$ and hence we have
\begin{equation} 
\label{eq:Useful_Relationship_2}
\begin{split}
 \big(\sup_{s\in[0,t]}\|\bar{\phi}_i(s,x_i,(\phi_{\neq i}, u))\|_{X_i}\big)_{i=1}^n
        &=
				 \big(\sup_{s\in[0,t]}\|\phi_i(s,x, u)\|_{X_i}\big)_{i=1}^n\\
        &=
				 \big(\|\phi_{i,[0,t]} \|_\infty\big)_{i=1}^n
        =:
\supnorm{\phi}{[0,t]}.                
\end{split}
\end{equation}


By assumptions, on $[0,t]$ the estimate \eqref{eq:UGS subsysteme} is valid. 
Taking in this estimate the supremum over $[0,t]$, and making use of \eqref{eq:Useful_Relationship_2}, we see that 
  \begin{equation}
    \label{eq:UGS subsysteme_2}
            \supnorm{\phi}{[0,t]} \leq \vec{\sigma}(\einsnorm{\bar{\phi}}{0,x,u} ) + \Gamma^{UGS}_\boxplus(\supnorm{\phi}{[0,t]})
    + \vec{\gamma}(\|u\|_\Uc),
  \end{equation}
    and thus 
  \begin{equation}
    \label{eq:UGS subsysteme_3}
    (I - \Gamma^{UGS}_\boxplus)\big(\supnorm{\phi}{[0,t]}\big)\leq \vec{\sigma}(\einsnorm{\bar{\phi}}{0,x,u}) + \vec{\gamma}(\|u\|_\Uc).
  \end{equation}
As $\Gamma^{UGS}_\boxplus$ is a monotone operator satisfying the strong small-gain condition, by Lemma~\ref{lem:boundedness_ID-Gamma_general case} there is a $\xi \in\Kinf$ so that 
\begin{eqnarray}
  \label{eq:UGS subsysteme_4}
  | \supnorm{\phi}{[0,\tau]}| &\leq& \xi\big(\big| \vec{\sigma}(\einsnorm{\bar{\phi}}{0,x,u}) + \vec{\gamma}(\|u\|_\Uc)\big|\big) \nonumber\\
                                                     &\leq& \xi\big( 2\big|\vec{\sigma}(\einsnorm{\bar{\phi}}{0,x,u})\big|\big) 
                                                    + \xi\big(2\big|\vec{\gamma}(\|u\|_\Uc)\big|\big).
\end{eqnarray}
We have:
\begin{eqnarray*}
\big|\vec{\sigma}(\einsnorm{\bar{\phi}}{0,x,u})\big|
=
\Big(\textstyle\sum_{i=1}^n \big(\sigma_i(\|x_i\|_{X_i})\big)^2\Big)^{\sfrac{1}{2}}
\leq 
\Big(\textstyle\sum_{i=1}^n \big(\sigma_i(\|x\|_{X})\big)^2\Big)^{\sfrac{1}{2}}
=:\sigma(\|x\|_X).
\end{eqnarray*}
Clearly, $\sigma\in\Kinf$. Furthermore, it holds that
\begin{eqnarray}
\label{eq:Useful_Relationship}
\phantom{aaaa}
\|\phi (t, x, u)\|_{X} = \Big(\textstyle\sum_{i=1}^n \|\phi_i(t, x, u)\|_{X_i}^2\Big)^{\sfrac{1}{2}}
\leq  \Big(\textstyle\sum_{i=1}^n \|\phi_{i,[0,t]} \|_\infty^2\Big)^{\sfrac{1}{2}}
=  | \supnorm{\phi}{[0,\tau]}|.
\end{eqnarray}
Substituting these expressions into \eqref{eq:UGS subsysteme_4}, we see that a UGS estimate
\begin{eqnarray}
  \label{eq:UGS subsysteme_5}
\|\phi (t, x, u)\|_{X} &\leq& \xi\big( 2\sigma(\|x\|_X)\big) + \xi\big(2\big|\vec{\gamma}(\|u\|_\Uc)\big|\big)
\end{eqnarray}
is valid on the certain maximal interval of existence $[0,t_m)$ of $\phi(\cdot,x,u)$.
As we assume that $\Sigma$ possesses the BIC property, then $\Sigma$ is forward complete, see Lemma~\ref{lem:ISS_UGS_and_BIC_property}.
\end{proof}

\subsection{ISS small-gain theorem}
\label{sec:ISS_SGT}

We start with a technical lemma:
\begin{lemma}
\label{lem:LimSupEstimate}
Let $g:\R_+\to\R^p_+$, $p\in\N$ be a globally bounded function and let $f:\R_+\to\R_+$ be an unbounded monotonically increasing function.
Then
\begin{eqnarray}
\lim_{t\to\infty} \sup_{s\geq f(t)} g(s)  =  \lim_{t\to\infty} \sup_{s\geq t} g(s). 
\label{eq:LimSupEstimate}
\end{eqnarray}
\end{lemma}
\begin{proof}
Define $a:=\lim_{t\to\infty} \sup_{s\geq f(t)} g(s)$ and $b:=\lim_{t\to\infty} \sup_{s\geq t} g(s)$. 
As $g$ is globally bounded, both $a$ and $b$ are well-defined and finite.
By definition it follows that for all $\varepsilon>0$ there is a time $T>0$ so that 
$\sup_{s\geq f(t)} g(s)  < a + \varepsilon$ for all $t\geq T$.
As $f$ is monotone, this is equivalent to the fact that 
\[
t\geq f(T) \qrq   \textstyle\sup_{s\geq t} g(s)  < a + \varepsilon.
\] 
As $\varepsilon$ can be chosen arbitrarily small, this shows that $b \leq a$.

Conversely, we have that for all $\varepsilon>0$ there is a time $T>0$ so that 
\begin{eqnarray}
t\geq T \qrq   \sup_{s\geq t} g(s)  < b + \varepsilon.
\label{eq:LimSup_Interm_Implication}
\end{eqnarray}
As $f$ is unbounded, there is a time $\tau=\tau(T)$ so that $f(\tau)>T$. 
Thus, \eqref{eq:LimSup_Interm_Implication} shows that $ \sup_{s\geq f(\tau)} g(s)  < b + \varepsilon$ and
as $f$ is monotone it follows that
$\sup_{s\geq f(t)} g(s)  < b + \varepsilon$
for all $t\geq \tau$. 
This implies that $a\leq b$. Overall, $a=b$.
\end{proof}

Now we are able to prove the main result of this paper.
\begin{theorem}[ISS Small-gain theorem]
\label{thm:ISS_SGT} 
Let $\Sigma_i:=(X_i,PC_b(\R_+,X_{\neq i}) \tm \Uc,\bar{\phi}_i)$, $i=1,\ldots,n$ be control systems, where all $X_i$, $i=1,\ldots,n$ and $\Uc$ are normed  linear spaces.
Assume that $\Sigma_i$, $i=1,\ldots,n$ are forward complete systems, satisfying the ISS estimates as in Lemma~\ref{lem:ISS_reformulation_n_systems}, and that the interconnection $\Sigma=(X,\Uc,\phi)$ is well-defined and possesses the BIC property.

If $\Gamma^{ISS}_\boxplus$ satisfies the strong small gain condition \eqref{eq:Strong_SGC}, then $\Sigma$ is ISS.    
\end{theorem}

\begin{proof}
We show UGS and bUAG properties of the interconnection $\Sigma=(X,\Uc,\phi)$, and infer ISS of $\Sigma$ by Lemma~\ref{lem:UGS_und_bUAG_imply_UAG}.

\noindent\textbf{UGS.} From the assumptions of the theorem, it follows that all $\Sigma_i$ are UGS with the gain matrix $\Gamma^{ISS}$. As $\Gamma^{ISS}_\boxplus$ satisfies the strong small gain condition, Theorem~\ref{thm:UGS_SGT} shows that the coupled system $\Sigma$ is forward complete and UGS.

\noindent\textbf{bUAG.} We use the notation \eqref{eq:phi_neq_i} for $\phi_i$ and $\phi_{\neq i}$.
As $\Sigma$ is a well-defined and forward complete interconnection, we have that $\phi_i(t,x,u) = \bar{\phi}_i(t,x_i,(\phi_{\neq i}, u))$ for all $i=1,\ldots,n$ and for all $t\geq 0$.

Pick any $r>0$, any $u\in\Uc$: $\|u\|_\Uc\leq r$ and any $x \in B_r$.
As $\Sigma$ is UGS, the estimate \eqref{GSAbschaetzung} is valid for some $\sigma_{UGS},\gamma_{UGS}\in\Kinf$ and it holds that
\begin{eqnarray}
\|\phi(t,x,u)\|_X \leq \mu(r):=\sigma_{UGS}(r) + \gamma_{UGS}(r),\quad t\geq 0.
\label{eq:UGS_our_case}
\end{eqnarray}
On the other hand, for all $i=1,\ldots,n$ due to the cocycle property, for any $t,\tau\geq 0$ we have
\begin{eqnarray}
\phi_i(t+\tau,x, u) &=&\bar{\phi}_i(t+\tau,x_i,(\phi_{\neq i}, u))  \nonumber\\
&=& \bar{\phi}_i\Big(\tau,\bar{\phi}_i\big(t,x_i,(\phi_{\neq i}, u)\big),\big(\phi_{\neq i}(\cdot + t), u(\cdot + t)\big)\Big).
\label{eq:Cocycle_property}
\end{eqnarray}
Note that due to the axiom of shift invariance, it holds that $u(t+\cdot) \in\Uc$.

In view of \eqref{eq:UGS_our_case} we have for all $i=1,\ldots,n$ that
\begin{eqnarray}
\|\bar{\phi}_i(t,x_i,\phi_{\neq i}, u)\|_{X_i} = \|\phi_i(t,x, u)\|_{X_i} \leq \|\phi(t,x, u)\|_{X} \leq \mu(r)
\label{eq:Bound_on_state}
\end{eqnarray}
and
\begin{eqnarray}
\|\phi_{\neq i}\|_\infty \leq \mu(r).
\label{eq:Bound_on_internal_inputs}
\end{eqnarray}
Furthermore, as $\sigma_{UGS}(r)\geq r$ for all $r\in\R_+$ (this follows from \eqref{GSAbschaetzung} by setting $u:=0$ and $t:=0$),
we have also
\begin{eqnarray}
\|u\|_\Uc \leq \mu(r).
\label{eq:Bound_on_external_inputs}
\end{eqnarray}
By the bUAG property of $\Sigma_i$ (which is implied by ISS of $\Sigma_i$), there is a time $\tau_i=\tau_i(\varepsilon,\mu(r))$ so that 
\begin{align}
\|x\|_X\leq r\ \wedge\ \|u\|_{\Uc}\leq r \ \wedge\ \tau\geq \tau_i &  \nonumber\\
 \qrq   \|\phi_i(t+\tau,x , u)\|_{X_i} 
&\leq \varepsilon + \textstyle\sum_{j\neq i}\gamma_{ij}\left(\left\|\phi_{j,[t,+\infty)}\right\|_{\infty}\right) + \gamma_i(\|u_{[t,+\infty)}\|_{\Uc})
\nonumber\\
&\leq \varepsilon + \textstyle\sum_{j\neq i}\gamma_{ij}\left(\left\|\phi_{j,[t,+\infty)}\right\|_{\infty}\right) + \gamma_i(\|u\|_{\Uc}).
\label{eq:ISS_SGT_tmp_estimate_1}
\end{align}
Note that dependence of $\tau_i$ on $r$ and $\varepsilon$ only (and not on $x$, $u$ and $t$) follows from the estimates \eqref{eq:Bound_on_state}, \eqref{eq:Bound_on_internal_inputs} and \eqref{eq:Bound_on_external_inputs}.

Define the uniform convergence time for all subsystems, which is finite, as we have finitely many subsystems.
\begin{eqnarray}
\tau^*(\varepsilon,r):=\max_{i=1,\ldots,n} \tau_i(\varepsilon,\mu(r)).
\label{eq:Uniform_time_UAG_subsystems}
\end{eqnarray}
Pick any $k \in\N$. 
Taking supremum of \eqref{eq:ISS_SGT_tmp_estimate_1} over $x\in B_r$ and over all $u\in\Uc$: $\|u\|_{\Uc} \in [2^{-k}r,2^{1-k}r]$, we obtain for all $i=1,\ldots,n$ and all $\tau\geq \tau^* $ that 
\begin{align}
\sup_{\|u\|_{\Uc} \in [2^{-k}r,2^{1-k}r]}&\sup_{x\in B_r}\|\phi_i(t+\tau,x , u)\|_{X_i} \nonumber\\
&\leq \varepsilon + \textstyle\sum_{j\neq i}\gamma_{ij}\Big(\sup_{\|u\|_{\Uc} \in [2^{-k}r,2^{1-k}r]}\sup_{x\in B_r}\left\|\phi_{j,[t,+\infty)}\right\|_{\infty}\Big) + \gamma_i(2^{1-k}r) 
\label{eq:ISS_SGT_tmp_estimate_2}
\end{align}
and thus
\begin{align}
\sup_{s\geq t+\tau^*}& \sup_{\|u\|_{\Uc} \in [2^{-k}r,2^{1-k}r]} \sup_{x\in B_r}\|\phi_i(s, x , u)\|_{X_i} \nonumber\\
&\leq \varepsilon + \textstyle\sum_{j\neq i}\gamma_{ij}\Big(\sup_{\|u\|_{\Uc} \in [2^{-k}r,2^{1-k}r]}\sup_{x\in B_r} \sup_{s\geq t} \|\phi_j(s, x , u)\|_{X_j}\Big) + \gamma_i(2^{1-k}r)\nonumber \\
&= \varepsilon + \textstyle\sum_{j\neq i}\gamma_{ij}\Big(\sup_{s\geq t} \sup_{\|u\|_{\Uc} \in [2^{-k}r,2^{1-k}r]} \sup_{x\in B_r} \|\phi_j(s, x , u)\|_{X_j}\Big) + \gamma_i(2^{1-k}r).
\label{eq:ISS_SGT_tmp_estimate_3}
\end{align}
Define 
\begin{align*}
y_i(r,k):=&  \lim_{t\to +\infty}\sup_{s\geq t}\sup_{\|u\|_{\Uc} \in [2^{-k}r,2^{1-k}r]}\sup_{x\in B_r}\|\phi_i(s, x , u)\|_{X_i}\\
=& \limsup_{t\to +\infty} \sup_{\|u\|_{\Uc} \in [2^{-k}r,2^{1-k}r]}\sup_{x\in B_r}\|\phi_i(t, x, u)\|_{X_i}.
\end{align*}
By Lemma~\ref{lem:LimSupEstimate} it holds that 
\[
y_i(r,k) =  \lim_{t\to +\infty}\sup_{s\geq t +\tau^*}\sup_{\|u\|_{\Uc} \in [2^{-k}r,2^{1-k}r]}\sup_{x\in B_r}\|\phi_i(s, x , u)\|_{X_i},
\]
and thus taking a limit $t\to\infty$ in \eqref{eq:ISS_SGT_tmp_estimate_3},     we have that
\begin{eqnarray}
y_i(r,k) &\leq& \varepsilon + \textstyle\sum_{j\neq i}\gamma_{ij}\left(y_j(r,k)\right) + \gamma_i(2^{1-k}r).
\label{eq:ISS_SGT_tmp_estimate_4}
\end{eqnarray}
As \eqref{eq:ISS_SGT_tmp_estimate_4} is valid for any $\varepsilon>0$, we obtain by computing the limit $\varepsilon\to +0$ that
\begin{eqnarray}
y_i(r,k) &\leq& \textstyle\sum_{j\neq i}\gamma_{ij}\left(y_j(r,k)\right) + \gamma_i(2^{1-k}r),\quad i=1,\ldots,n.
\label{eq:ISS_SGT_tmp_estimate_5}
\end{eqnarray}
With the notation \eqref{eq:vector_gamma} and by defining 
$y(r,k):=(y_1(r,k),\ldots, y_n(r,k))^T \in\R^n_+$ we can rewrite \eqref{eq:ISS_SGT_tmp_estimate_5} in a vector form:
\begin{eqnarray}
y(r,k) &\leq& \Gamma^{ISS}_\boxplus\left(y(r,k)\right) + \vec{\gamma}(2^{1-k}r).
\label{eq:ISS_SGT_tmp_estimate_6}
\end{eqnarray}
Since $\Gamma^{ISS}_\boxplus$ satisfies the strong small-gain condition, applying Lemma~\ref{lem:boundedness_ID-Gamma_general case} to the inequality $(\id-\Gamma^{ISS}_\boxplus)(y(r,k)) \leq \vec{\gamma}(2^{1-k}r)$, we obtain that there is $\xi\in\Kinf$ so that 
\begin{eqnarray*}
|y(r,k)| &\leq& \xi(|\vec{\gamma}(2^{1-k}r)|).
\end{eqnarray*}
Using a technical computation
\begin{align*}
&\limsup_{t\to +\infty} \sup_{\|u\|_{\Uc} \in [2^{-k}r,2^{1-k}r]}\sup_{x\in B_r}\|\phi(t, x, u)\|_{X}  \\
&=
\limsup_{t\to +\infty}\sup_{\|u\|_{\Uc} \in [2^{-k}r,2^{1-k}r]}\sup_{x\in B_r}\Big(\textstyle\sum_{i=1}^n \|\phi_i(t, x , u)\|_{X_i}^2\Big)^{\sfrac{1}{2}}\\
&\leq
\Big(\sum_{i=1}^n \limsup_{t\to +\infty}\sup_{\|u\|_{\Uc} \in [2^{-k}r,2^{1-k}r]}\sup_{x\in B_r} \|\phi_i(t, x , u)\|_{X_i}^2\Big)^{\sfrac{1}{2}}
\hspace{-2mm}=
\Big(\sum_{i=1}^n y^2_i(r,k)\Big)^{\sfrac{1}{2}}
\hspace{-2mm}=
|y(r,k)|,
\end{align*}
we conclude, that for any $r>0$ and any $k\in\N$ it holds that
\begin{eqnarray*}
\limsup_{t\to +\infty}\sup_{\|u\|_{\Uc} \in [2^{-k}r,2^{1-k}r]}\sup_{x\in B_r}\|\phi(t, x, u)\|_{X}  \leq \xi(|\vec{\gamma}(2^{1-k}r)|).
\end{eqnarray*}
Hence, for any $\varepsilon>0$, any $r>0$ and any $k\in\N$ there is a time $\tilde{\tau}=\tilde{\tau}(\varepsilon,r,k)$ so that

\begin{equation} 
\label{eq:ISS_SGT_interediate_step}
\begin{split}
\|x\|_X\leq r\ \wedge \ \|u\|_\Uc \in [2^{-k}r,2^{1-k}r] &\ \wedge \ t\geq\tilde{\tau}(\varepsilon,r,k)\\
& \qrq \|\phi(t, x, u)\|_{X}  \leq \varepsilon + \xi(|\vec{\gamma}(2^{1-k}r)|).
\end{split}
\end{equation}
%
Define $k_0=k_0(\varepsilon,r) \in\N$ as the minimal $k$ so that $\xi(|\vec{\gamma}(2^{1-k}r)|) \leq \varepsilon$ (clearly, such $k_0$ always exists and is finite) and let
\[
\hat{\tau}(\varepsilon,r):=\max\{\tilde{\tau}(\varepsilon,r,k):\ k= 1,\ldots,k_0(\varepsilon,r)\}.
\]
As $k_0$ is finite, $\hat{\tau}(\varepsilon,r)$ is well-defined and finite as well.

Pick any $u\in\Uc$ such that $u\neq 0$ and  $\|u\|_{\Uc}\leq r$. 
Then there is $k\in\N$ so that
$\|u\|_\Uc \in (2^{-k}r,2^{1-k}r]$.
If $k\leq k_0$ (i.e. if inputs are large enough), then for $t\geq\hat{\tau}(\varepsilon,r) $ it holds that 
\begin{eqnarray}
\|\phi(t, x, u)\|_{X}  \leq \varepsilon + \xi(|\vec{\gamma}(2^{1-k}r)|) \leq \varepsilon + \xi(|\vec{\gamma}(2 \|u\|_\Uc)|).
\label{eq:Final_bUAG_implication_1}
\end{eqnarray}
It remains to consider the case when $k>k_0$, i.e. when inputs are small.
The estimate \eqref{eq:ISS_SGT_interediate_step} gives convergence time, which depends on $k$ and it is not clear whether the supremum of
$\tilde{\tau}(\varepsilon,r,k)$ over all $k\geq k_0$ exists.
In order to overcome this obstacle and to find the uniform time, we mimic above argument once again, namely: for any $q \in [0,r]$ one can take  
supremum of \eqref{eq:ISS_SGT_tmp_estimate_1} over $x\in B_r$ and over all $u\in\Uc$: $\|u\|_{\Uc} \leq q$, to obtain for all $i=1,\ldots,n$ and all $\tau\geq \tau^* $ that 
\begin{equation} 
\label{eq:ISS_SGT_tmp_estimate_Small_U}
\sup_{\|u\|_{\Uc}\leq q}\sup_{x\in B_r}\|\phi_i(t+\tau,x , u)\|_{X_i} 
\leq \varepsilon + {\textstyle\sum_{j\neq i}}\gamma_{ij}\big(\hspace{-1mm}\sup_{\|u\|_{\Uc} \leq q}\sup_{x\in B_r}\left\|\phi_{j,[t,+\infty)}\right\|_{\infty}\hspace{-1mm}\big) {+} \gamma_i(q),
\end{equation}
where $\tau^*$ has been defined in \eqref{eq:Uniform_time_UAG_subsystems}. Defining 
\[
z_i(r,q):=  \lim_{t\to +\infty}\sup_{s\geq t}\sup_{\|u\|_{\Uc} \leq q}\sup_{x\in B_r}\|\phi_i(s, x , u)\|_{X_i}
\]
and doing analogous steps as above we obtain for any $r>0$ and any $q\leq r$ that
\begin{eqnarray*}
\limsup_{t\to +\infty}\sup_{\|u\|_{\Uc} \leq q}\sup_{x\in B_r}\|\phi(t, x, u)\|_{X}  \leq \xi(|\vec{\gamma}(q)|).
\end{eqnarray*}
This means that for any $\varepsilon>0$, any $r>0$ and any $q\geq 0$ there is a time $\bar{\tau}=\bar{\tau}(\varepsilon,r,q)$ so that
\begin{eqnarray}
\phantom{aaaaa}\|x\|_X\leq r\ \wedge \ \|u\|_\Uc \leq q \ \wedge \ t\geq\bar{\tau}(\varepsilon,r,q) \srs \|\phi(t, x, u)\|_{X} \leq \varepsilon + \xi(|\vec{\gamma}(q)|).
\label{eq:ISS_SGT_interediate_step_Small_U}
\end{eqnarray}
In particular, for $q_0:=2^{1-k_0(\varepsilon,r)}r$ we have that 
\begin{equation} 
\label{eq:Final_bUAG_implication_2}
\begin{split}
\|x\|_X\leq r\ \wedge \ \|u\|_\Uc \leq &q_0 \ \wedge \ t\geq\bar{\tau}(\varepsilon,r,q_0) \\
& \qrq \|\phi(t, x, u)\|_{X}  \leq \varepsilon + \xi(|\vec{\gamma}(2^{1-k_0(\varepsilon,r)}r)|)\leq \varepsilon + \varepsilon.
\end{split}
\end{equation}
%
Combining \eqref{eq:Final_bUAG_implication_1} and \eqref{eq:Final_bUAG_implication_2}, we obtain for 
$\tau(\varepsilon,r) :=\max\{\hat{\tau}(\varepsilon,r), \bar{\tau}(\varepsilon,r,q_0) \}$ that
\[
\|x\|_X\leq r\ \wedge \|u\|_\Uc \leq r \ \wedge t\geq\tau(\varepsilon,r) \qrq  \|\phi(t, x, u)\|_{X} 
\leq \varepsilon + \max\{\varepsilon, \xi(|\vec{\gamma}(2 \|u\|_\Uc)|) \}.
\]
and finally
\[
\|x\|_X\leq r\ \wedge \|u\|_\Uc \leq r \ \wedge t\geq\tau(\varepsilon,r) \qrq \|\phi(t, x, u)\|_{X} 
\leq 2\varepsilon + \xi(|\vec{\gamma}(2 \|u\|_\Uc)|).
\]
As $r\mapsto \xi(|\vec{\gamma}(2 r)|)$ is a $\Kinf$-function, this implication shows that $\Sigma$ is bUAG.

\noindent
\textbf{ISS.} Since $\Sigma$ is UGS $\wedge$ bUAG, Lemma~\ref{lem:UGS_und_bUAG_imply_UAG} implies ISS of $\Sigma$.
\end{proof}

\begin{remark}[Discussion of the proof of Theorem~\ref{thm:ISS_SGT}]
\label{rem:ISS_SGT_Proof_Technique} 
A key step in the proof of Theorem~\ref{thm:ISS_SGT} is the shifting of the time horizon, see \eqref{eq:ISS_SGT_tmp_estimate_1}, achieved by means of the cocycle property \eqref{eq:ISS_SGT_tmp_estimate_3}. 
It is important that we want to achieve the dependence of the convergence time $\tau_i$ on $r$ and $\varepsilon$ only, which follows from
\eqref{eq:Bound_on_state}, \eqref{eq:Bound_on_internal_inputs} and \eqref{eq:Bound_on_external_inputs}, which are valid in turn since we consider the inputs with a norm uniformly bounded by $r$. Having an arbitrary $u$ would result that the norm of the state $\phi_i(t,x_i,\phi_{\neq i}, u)$ in \eqref{eq:Cocycle_property}, and hence the time $\tau_i$ would depend on the norm of $u$, which makes it hard to obtain UAG property of the interconnection.
Using bUAG property instead helps us to avoid these complications and this is one of the reasons for introducing this property.

Then, in order to tackle the distinction in the time intervals over which the supremum in the left and right-hand sides of \eqref{eq:ISS_SGT_tmp_estimate_3} is taken, we pass to the limit $t\to\infty$. However, before we can compute such limit we have to make the expressions in both sides of \eqref{eq:ISS_SGT_tmp_estimate_1} independent on $x,u$. Taking supremum over $x\in B_r$ causes no problems, but 
taking a supremum over $u \in B_{r,\Uc}$ leads to overly rough estimates, and thus we slice the ball $u \in B_{r,\Uc}$
into several \q{rings}, which helps in the end to obtain the desired bUAG property of the whole interconnection.
\end{remark}

The discussion of the obtained results and their relations to the existing works we postpone to the Section~\ref{sec:Discussion}, and proceed to the weak ISS small-gain theorems.

\subsection{AG and weak ISS small-gain theorem}
\label{sec:wISS_SGT}

Our next result is the small-gain theorem for the asymptotic gain property. 

\begin{theorem}[AG Small-gain theorem]
\label{thm:AG_SGT} 
Let $\Sigma_i:=(X_i,PC_b(\R_+,X_{\neq i}) \tm \Uc,\bar{\phi}_i)$, $i=1,\ldots,n$ be control systems, where all $X_i$, $i=1,\ldots,n$ and $\Uc$ are normed  linear spaces.
Assume that $\Sigma_i$, $i=1,\ldots,n$ are forward complete, satisfy the AG estimates as in Lemma~\ref{lem:AG_reformulation_n_systems}, and that the interconnection $\Sigma$ is well-defined and forward complete.

If $\Gamma^{AG}_\boxplus$ satisfies the strong small gain condition \eqref{eq:Strong_SGC}, then $\Sigma$ is AG.
\end{theorem}

In this case, the complexities, described in Remark~\ref{rem:ISS_SGT_Proof_Technique} do not appear, and the proof goes along the lines of the proof of Theorem~\ref{thm:ISS_SGT}, with significant simplifications. Hence in the following rather sketchy argument, we merely indicate the main differences to the detailed proof of Theorem~\ref{thm:ISS_SGT}. 
It is also possible to show Theorem~\ref{thm:AG_SGT} along the lines of the proof of the corresponding finite-dimensional counterpart \cite[Theorem 9]{DRW07}.

\begin{proof}[Sketch]
%
Pick any $x \in X$ and any $u\in\Uc$ and define $\phi_i$ and $\phi_{\neq i}$ as in \eqref{eq:phi_neq_i}. The cocycle property \eqref{eq:Cocycle_property} and AG property of $\Sigma_i$ imply existence of a time $\tau_i=\tau_i(\varepsilon,x,u,t)$ which we can assume to be an increasing function of $t$, such that 
\begin{eqnarray}
  \sup_{\tau\geq \tau_i+t}\|\phi_i(\tau,x , u)\|_{X_i}& =&   \sup_{\tau\geq \tau_i}\|\phi_i(t+\tau,x , u)\|_{X_i} \nonumber\\
&\leq& \varepsilon + \sum_{j\neq i}\gamma_{ij}\left(\left\|\phi_{j,[t,+\infty)}\right\|_{\infty}\right) + \gamma_i(\|u\|_{\Uc}).
\label{eq:AG_SGT_tmp_estimate_1}
\end{eqnarray}
Note that the time $\tau_i$ depends on the tuple $(\varepsilon,\phi_i(t,x_i,\phi_{\neq i}, u),\phi_{\neq i}(\cdot + t), u(\cdot + t))$ 
in \eqref{eq:Cocycle_property}, but all these parameters depend on $(\varepsilon,x,u,t)$ only.

{
Defining $\tau^*(\varepsilon,x,u,t):=\displaystyle\max_{i=1,\ldots,n} \tau_i(\varepsilon,x,u,t)$,\quad $y_i(x,u):=  \displaystyle\lim_{t\to +\infty}\sup_{s\geq t}\|\phi_i(s, x , u)\|_{X_i}$,  $y(r,k):=(y_1(r,k),\ldots, y_n(r,k))^T \in\R^n_+$, and doing the same steps as in the proof of ISS small-gain theorem
we obtain
}
that there is $\xi\in\Kinf$ so that 
\begin{eqnarray}
\limsup_{t\to +\infty} \|\phi(s, x, u)\|_{X} = |y(r,k)| &\leq& \xi(|\vec{\gamma}(\|u\|_{\Uc})|),
\label{eq:AG_SGT_tmp_estimate_7}
\end{eqnarray}
which is precisely the asymptotic gain property of $\Sigma$.
\end{proof}

\begin{remark}
\label{rem:FC_is_needed_for_AG_SGT} 
A notable difference in the assumptions of the AG small-gain theorem, in compare to UGS and ISS small-gain theorems is that 
the forward completeness of the interconnection is required.
The assumption of forward-completeness cannot be relaxed to merely BIC property, as demonstrated in \cite[Section 4.2]{DRW07} already for couplings of ODE systems.
\end{remark}

\begin{definition}
\label{def:weak_ISS} 
A control system $\Sigma$ possessing the BIC property is called weakly ISS, provided $\Sigma$ is UGS and possesses AG property.
It is then forward complete, see Lemma~\ref{lem:ISS_UGS_and_BIC_property}.
\end{definition}

For ODEs, weak ISS is equivalent to ISS, but it is much weaker than ISS even for linear infinite-dimensional systems. 

As a combination of UGS and AG small-gain theorems we obtain 
\begin{theorem}[Weak ISS Small-gain theorem]
\label{thm:wISS_SGT} 
Let $\Sigma_i:=(X_i,PC_b(\R_+,X_{\neq i}) \tm \Uc,\bar{\phi}_i)$, $i=1,\ldots,n$ be control systems, where all $X_i$, $i=1,\ldots,n$ and $\Uc$ are normed  linear spaces.
Assume that $\Sigma_i$, $i=1,\ldots,n$ are forward complete weakly ISS systems with the same gain matrices $\Gamma^{wISS}$ for both AG and UGS property (formulated as in Lemma~\ref{lem:UGS_reformulation_n_systems}, Lemma~\ref{lem:AG_reformulation_n_systems}), and that the interconnection $\Sigma$ is well-defined and possesses the BIC property.

If $\Gamma^{wISS}_\boxplus$ satisfies the strong small gain condition \eqref{eq:Strong_SGC}, then $\Sigma$ is weakly ISS.    
\end{theorem}

\begin{proof}
As each subsystem of $\Sigma$ is UGS, $\Sigma$ possesses the BIC property, and $\Gamma^{wISS}_\boxplus$ is a gain operator for UGS property, 
Theorem~\ref{thm:UGS_SGT} shows that $\Sigma$ is forward complete and UGS.
As weakly ISS systems possess an AG property, $\Sigma$ is forward complete and 
$\Gamma^{wISS}_\boxplus$ is a gain operator for AG property satisfying the small-gain condition, 
Theorem~\ref{thm:AG_SGT} implies that $\Sigma$ has AG property.
Thus, $\Sigma$ is weakly ISS. 
\end{proof}

\subsection{ISS small-gain theorem in semimaximum formulation}
\label{sec:Semimaximum_formulation}

In Lemma~\ref{lem:ISS_reformulation_n_systems} we have reformulated the ISS property in a way that the total influence of subsystems is the sum of the internal gains. In some cases, this formulation is the most convenient, but in other cases, other restatements can be more useful.

Another important restatement, which mixes summation and maximization, is given in the next lemma (we again omit the proof of this result):
\begin{lemma}
\label{lem:ISS_reformulation_n_systems_semimax_form}
A forward complete system \emph{$\Sigma_i$ is ISS (in semimaximum formulation)} if there exist $\gamma_{ij}, \gamma_i \in \Kinf\cup\{0\},\ j=1,\ldots,n$ and $\beta_i \in \KL$, such that for all initial values $x_i \in X_i$, all internal inputs 
$w_{\neq i} := (w_1,\ldots,w_{i-1}, w_{i+1},\ldots,w_n) \in PC_b(\R_+,X_{\neq i})$, all external inputs $u \in\Uc$ and all $t \in\R_+$ the following holds:
\begin{eqnarray}
\label{eq:ISS_n_sys_max_semi-sum_form}
\phantom{aaaaaa}\|\bar{\phi}_i\big(t,x_i,(w_{\neq i}, u)\big)\|_{X_i}  
\leq \beta_i\left(\left\|x_i\right\|_{X_i},t\right) +  \max_{j\neq i}\big\{\gamma_{ij}\big(\left\|w_j\right\|_{[0,t]}\big)\big\} + \gamma_i\left(\left\|u\right\|_{\Uc}\right).
\end{eqnarray}
\end{lemma}
%

We can collect all the internal gains $\gamma_{ij}$ from the semimaximum reformulation \eqref{eq:ISS_n_sys_max_semi-sum_form} 
of ISS again into the matrix $\Gamma$ and introduce instead of the operator $\Gamma_\boxplus$ the operator $\Gamma_\otimes:\R^n_+\to\R^n_+$ acting for $s=(s_1,\ldots,s_n)^T\in\R^n_+$ as
\begin{equation}
\Gamma_\otimes(s) := \Big(  \max_{j=1}^n \gamma_{1j}(s_j),\ldots,\max_{j=1}^n \gamma_{nj}(s_j)  \Big)^T.
\label{eq:Gamma-max}
\end{equation}

A counterpart of the ISS small-gain Theorem~\ref{thm:ISS_SGT} for the semimaximum formulation of the ISS property is given by the next result:
\begin{theorem}[ISS Small-gain theorem in semimaximum formulation]
\label{thm:ISS_SGT_semimax_form} 
Let $\Sigma_i:=(X_i,PC_b(\R_+,X_{\neq i}) \tm \Uc,\bar{\phi}_i)$, $i=1,\ldots,n$ be control systems, where all $X_i$, $i=1,\ldots,n$ and $\Uc$ are normed  linear spaces.
Assume that $\Sigma_i$, $i=1,\ldots,n$ are forward complete systems, satisfying the ISS estimates as in Lemma~\ref{lem:ISS_reformulation_n_systems_semimax_form}, and that the interconnection $\Sigma=(X,\Uc,\phi)$ is well-defined and possesses the BIC property.

If $\Gamma^{ISS}_\otimes$ satisfies the strong small gain condition \eqref{eq:Strong_SGC}, then $\Sigma$ is ISS.    
\end{theorem}

\begin{proof}
Pick any $u\in\Uc$ and any $x\in X$. As the interconnection $\Sigma$ is well-defined, there is a certain $t_1 >0$ so that $(t,x,u)\in D_\phi$ for all $t \in [0,t_1]$.
Using the notation and the arguments introduced in the UGS small-gain theorem in the summation form (Theorem~\ref{thm:UGS_SGT}),
we obtain the following estimate for all $t\in[0,t_1]$:
\begin{equation*}
 \supnorm{\phi}{[0,t]} \leq \vec{\sigma}(\einsnorm{\bar{\phi}}{0,x,u} ) + \Gamma^{ISS}_\otimes(\supnorm{\phi}{[0,t]}) + \vec{\gamma}(\|u\|_\Uc),
\end{equation*}
and as $\Gamma^{ISS}_\otimes$ satisfies the strong small gain condition \eqref{eq:Strong_SGC},
we can again apply Lemma~\ref{lem:boundedness_ID-Gamma_general case} to show UGS property of the interconnection
as it was done in Theorem~\ref{thm:UGS_SGT}.

The rest of the proof is similar to the proof of Theorem~\ref{thm:ISS_SGT} and is omitted.
\end{proof}

\subsection{ISS small-gain theorem in maximum formulation}
\label{sec:Maximum_formulation_new}

Finally, consider a reformulation of UGS and ISS in terms of maximums only, which is frequently used in ISS theory:
\begin{lemma}
\label{lem:UGS_reformulation_n_systems_max_form}
A forward complete system \emph{$\Sigma_i$ is UGS (in maximum formulation)} if there exist $\gamma_{ij}, \gamma_i \in \Kinf\cup\{0\},\ j=1,\ldots,n$ and $\sigma_i \in \Kinf$, such that for all initial values $x_i \in X_i$, all internal inputs 
$w_{\neq i} := (w_1,\ldots,w_{i-1}, w_{i+1},\ldots,w_n) \in PC_b(\R_+,X_{\neq i})$, all external inputs $u \in\Uc$ and all $t \in\R_+$ the following holds:
\begin{eqnarray}
\label{eq:UGS_n_sys_max}
\|\bar{\phi}_i\big(t,x_i,(w_{\neq i}, u)\big)\|_{X_i}  \leq
 \max_{j\neq i}\big\{ \sigma_i\left(\left\|x_i\right\|_{X_i}\right), \gamma_{ij}\big(\left\|w_j\right\|_{[0,t]}\big), \gamma_i\left(\left\|u\right\|_{\Uc}\right)\big\}.
\end{eqnarray}
\end{lemma}

\begin{lemma}
\label{lem:ISS_reformulation_n_systems_max_form}
A forward complete system \emph{$\Sigma_i$ is ISS (in maximum formulation)} if there exist $\gamma_{ij}, \gamma_i \in \K\cup\{0\},\ j=1,\ldots,n$ and $\beta_i \in \KL$, such that for all initial values $x_i \in X_i$, all internal inputs 
$w_{\neq i} := (w_1,\ldots,w_{i-1}, w_{i+1},\ldots,w_n) \in PC_b(\R_+,X_{\neq i})$, all external inputs $u \in\Uc$ and all $t \in\R_+$ the following holds:
\begin{eqnarray}
\label{eq:ISS_n_sys_max}
\|\bar{\phi}_i\big(t,x_i,(w_{\neq i}, u)\big)\|_{X_i}  \leq
 \max_{j\neq i}\big\{ \beta_i\big(\left\|x_i\right\|_{X_i},t\big), \gamma_{ij}\big(\left\|w_j\right\|_{[0,t]}\big), \gamma_i\big(\left\|u\right\|_{\Uc}\big)\big\}.
\end{eqnarray}
\end{lemma}
The (componentwise) supremum of two elements $x,y \in \R^n$ we denote by $x \vee y$.

In the case of maximum formulation, one can relax the assumptions in UGS and ISS small-gain theorems and require only validity of the small gain condition 
\eqref{eq:SGC} (in contrast to the strong small gain condition \eqref{eq:Strong_SGC} in the case of summation and semimaximum formulations).
This relaxation is possibly due to the following criterion, shown in \cite[Theorem 6.4]{Rue10}, which replaces the role played by Lemma~\ref{lem:boundedness_ID-Gamma_general case} in the proof of sum-form ISS small-gain theorem.
\begin{proposition}
\label{prop:SGC-max-preserving-operators} 
The small gain condition \eqref{eq:SGC} holds with $A=\Gamma_\otimes$ if and only if
there is $\xi\in\Kinf$ such that for all $w,v\in \R^n_+$ it holds that 
\begin{eqnarray*}
w \leq \Gamma_\otimes(w) \vee v \srs |w| \leq \xi(|v|).
\end{eqnarray*}
\end{proposition}
With Proposition~\ref{prop:SGC-max-preserving-operators} in mind, we can state the following max-form UGS small-gain theorem:
\begin{theorem}[UGS Small-gain theorem: maximum formulation]
\label{thm:UGS_SGT_max_form} 
Let $\Sigma_i:=(X_i,PC_b(\R_+,X_{\neq i}) \tm \Uc,\bar{\phi}_i)$, $i=1,\ldots,n$ be control systems, where all $X_i$, $i=1,\ldots,n$ and $\Uc$ are normed  linear spaces.
Assume that $\Sigma_i$, $i=1,\ldots,n$ are forward complete systems, satisfying the UGS estimates as in Lemma~\ref{lem:UGS_reformulation_n_systems_max_form},
 and that the interconnection $\Sigma=(X,\Uc,\phi)$ is well-defined and possesses the BIC property.

If $\Gamma^{UGS}_\otimes$ satisfies the small gain condition \eqref{eq:SGC}, then $\Sigma$ is forward complete and UGS.
\end{theorem}

\begin{proof}
The proof is analogous to the proof of Theorem~\ref{thm:UGS_SGT}, thus we state only the main differences.
Pick any $u\in\Uc$ and any initial condition $x\in X$.
Using the notation of  Theorem~\ref{thm:UGS_SGT}, we obtain the estimate
  \begin{equation}
    \label{eq:UGS subsysteme_2-max-form}
            \supnorm{\phi}{[0,t]} \leq \Gamma^{UGS}_\otimes(\supnorm{\phi}{[0,t]}) \vee \vec{\sigma}(\einsnorm{\bar{\phi}}{0,x,u} )  \vee \vec{\gamma}(\|u\|_\Uc).
  \end{equation}
Using Proposition~\ref{prop:SGC-max-preserving-operators}, there is $\xi\in\Kinf$ such that the following holds:
\begin{eqnarray}
  \label{eq:UGS subsysteme_4_max-form}
  | \supnorm{\phi}{[0,\tau]}| &\leq& \xi\big(\big|\vec{\sigma}(\einsnorm{\bar{\phi}}{0,x,u} )  \vee \vec{\gamma}(\|u\|_\Uc)\big|\big) 
                                                     \leq \xi\big(\big|\vec{\sigma}(\einsnorm{\bar{\phi}}{0,x,u})\big| + \big|\vec{\gamma}(\|u\|_\Uc)\big|\big)\nonumber\\
                                                     &\leq& \xi\big(2\big|\vec{\sigma}(\einsnorm{\bar{\phi}}{0,x,u})\big|\big) +  \xi\big(2\big|\vec{\gamma}(\|u\|_\Uc)\big|\big).
\end{eqnarray}
Arguing as in Theorem~\ref{thm:UGS_SGT}, the claim follows.
\end{proof}

The max-form ISS small-gain theorem we state as follows:
\begin{theorem}[ISS Small-gain theorem: maximum formulation]
\label{thm:ISS_SGT_max_form} 
Let $\Sigma_i:=(X_i,PC_b(\R_+,X_{\neq i}) \tm \Uc,\bar{\phi}_i)$, $i=1,\ldots,n$ be control systems, where all $X_i$, $i=1,\ldots,n$ and $\Uc$ are normed  linear spaces.
Assume that $\Sigma_i$, $i=1,\ldots,n$ are forward complete systems, satisfying the ISS estimates as in Lemma~\ref{lem:ISS_reformulation_n_systems_max_form}, and that the interconnection $\Sigma=(X,\Uc,\phi)$ is well-defined and possesses the BIC property.

If $\Gamma^{ISS}_\otimes$ satisfies the small gain condition \eqref{eq:SGC}, then $\Sigma$ is ISS.    
\end{theorem}

\begin{proof}
The proof is analogous to the proof of Theorem~\ref{thm:ISS_SGT}, thus we state only the main differences.

\noindent\textbf{UGS.} Follows by Theorem~\ref{thm:UGS_SGT_max_form}. 

\noindent\textbf{bUAG.} Proceeding as in the proof of Theorem~\ref{thm:ISS_SGT} (and using the same notation), we obtain the counterpart of \eqref{eq:ISS_SGT_tmp_estimate_6} in the form:
\begin{eqnarray}
y(r,k) &\leq& \Gamma^{ISS}_\otimes\left(y(r,k)\right) \vee \vec{\gamma}(2^{1-k}r).
\label{eq:ISS_SGT_tmp_estimate_6-max-form}
\end{eqnarray}
As we assume that $\Gamma^{ISS}_\otimes$ satisfies the small-gain condition \eqref{eq:SGC}, Proposition~\ref{prop:SGC-max-preserving-operators} shows
that there is $\xi\in\Kinf$ so that 
\begin{eqnarray*}
|y(r,k)| &\leq& \xi(|\vec{\gamma}(2^{1-k}r)|).
\end{eqnarray*}
Now again arguing as in as in Theorem~\ref{thm:ISS_SGT}, we obtain bUAG property of $\Sigma$.

\noindent \textbf{ISS.} Since $\Sigma$ is UGS $\wedge$ bUAG, Lemma~\ref{lem:UGS_und_bUAG_imply_UAG} implies ISS of $\Sigma$.
\end{proof}

\subsection{Tightness of ISS small-gain theorems}
\label{sec:Tightness_small-gain-theorem}

Small-gain theorems in sum-form for UGS, ISS, AG and weak ISS properties have been derived in this paper under the assumption that the strong small-gain condition \eqref{eq:Strong_SGC} holds. A natural question is whether the same result holds under weaker conditions.

Consider the planar system considered in \cite[Example 18]{DRW07}: 
\begin{subequations}
  \begin{align}
\dot{x}_1 &=  -x_1 + x_2(1-e^{-x_2}) + u(t),\label{eq:1A} \\
\dot{x}_2 &=  -x_2 + x_1(1-e^{-x_1}) + u(t).\label{eq:1B} 
  \end{align}
\label{eq:SGC_not_Strong_SGC}
\end{subequations}
In \cite[Example 18]{DRW07} it was shown that:
\begin{itemize}
    \item Both subsystems of \eqref{eq:SGC_not_Strong_SGC} are ISS in a summation formulation with the gains $\gamma_{12}(r) = \gamma_{21}(r) := r(1-e^{-r})$, $r\in\R_+$
    \item The corresponding operator $\Gamma_\boxplus$, satisfies the small-gain condition \eqref{eq:SGC}, which boils down to checking that $\gamma_{12}\circ\gamma_{21}(r) <r$ for all $r>0$
    \item $\Gamma_\boxplus$ does not satisfy the strong small-gain condition
    \item \eqref{eq:SGC_not_Strong_SGC} is not ISS
\end{itemize}
Hence in the statement of the ISS small-gain theorem in the summation formulation the requirement of the strong small-gain condition for $\Gamma_\boxplus$ cannot be weakened to the requirement of a small-gain condition for $\Gamma_\boxplus$.

As for the interconnection of 2 systems, the summation and semimaximum formulations of the ISS property coincide and $\Gamma_\boxplus = \Gamma_\otimes$, 
the same example shows the tightness of Theorem~\ref{thm:ISS_SGT_semimax_form} in the above sense.
Note that although the system \eqref{eq:SGC_not_Strong_SGC} is not ISS, it is 0-UGAS, 
which can be shown using the Lyapunov function $V(x_1,x_2)=x_1^2 + x_2^2$. 

In the next proposition we show that for any gain matrix for which the corresponding gain operator $\Gamma_{\boxplus}$
does not satisfy the small-gain condition \eqref{eq:SGC}, one can construct a system with this gain matrix, so that each subsystem of this system is ISS in the summation formulation, but the interconnection is not 0-UGAS.
\begin{proposition}
\label{prop:Tightness_SGC}
Let a gain matrix $\Gamma:=(\gamma_{ij})_{i,j=1,\ldots,n} \subset (\K \cup\{0\})^{n\times n}$, with $\gamma_{ii}=0$ for all $i=1,\ldots,n$ be given. If the corresponding gain operator $\Gamma_\boxplus$ does not satisfy the small-gain condition \eqref{eq:SGC}, then there exists $f:\R^n \times \R^m \to \R^n$ so that all subsystems of the ODE system \eqref{eq:Coupled_n_ODEs} are ISS with the gains $(\gamma_{ij})_{i,j=1,\ldots,n}$, but the whole interconnection \eqref{eq:ODE_coupled_sys} has a non-trivial equilibrium and thus is not 0-UGAS.
\end{proposition}

\begin{proof}
Let $\Gamma_\boxplus$ does not satisfy the small-gain condition \eqref{eq:SGC}.
Then there is $s\in \R^n_+\backslash\{0\}$ so that $\Gamma_\boxplus(s) \geq s$.
Take $\varepsilon\geq 0$ so that 
\begin{eqnarray}
(1-\varepsilon)\Gamma_\boxplus(s) = s,
\label{eq:gain-equality}
\end{eqnarray}
and enlarge the domain of definition of functions $\gamma_{ij}$ to $\R$, defining $\gamma_{ij}(-r)=-\gamma_{ij}(r)$ $\forall r>0, \ i,j = 1,\ldots n$. With such definitions we can consider the operator $\Gamma_\boxplus$ as an operator acting on $\R^n$, with the same defining formula
\eqref{eq:ps_gamma}.

Consider the following ODE system on $\R^n$:
\begin{eqnarray}
\dot{x} = -x + (1-\varepsilon)\Gamma_\boxplus(x).
\label{KGB_Gegenbeispiel}
\end{eqnarray}
The point $x=s$ is a non-trivial equilibrium of \eqref{KGB_Gegenbeispiel}, and hence it is not 0-UGAS.

However, all subsystems of \eqref{KGB_Gegenbeispiel} satisfy the following estimates
\begin{eqnarray*}
|x_i(t)| &\leq& \left|x_i(0)\right|e^{-t} + e^{-t}\int^t_0{e^{s}\sum_{j=1}^n(1-\varepsilon)\gamma_{ij}(|x_j(s)|)ds}
\leq \left|x_i(0)\right|e^{-t} + \sum_{j=1}^n\gamma_{ij} \left(\|x_{j}\|_{\infty} \right).
\end{eqnarray*}
and hence are ISS with the gains $(\gamma_{ij})_{i,j=1,\ldots,n}$.
This shows the claim.
\end{proof}

\begin{remark}
\label{rem:Tightness_small-gain_condition} 
It is possible to show a counterpart of Proposition~\ref{prop:Tightness_SGC} for the semimaximum formulation of ISS, see \cite[Theorem 1.5.9]{Mir12} and \cite[Theorem 1.5.10]{Mir12} for the corresponding Lyapunov version.
\end{remark}

\subsection{Discussion and examples}
\label{sec:Discussion}

In this section, we discuss the main features of the obtained results and relate them to existing work. 

\subsubsection{Generality}
Small-gain theorems developed in this work are very general and applicable for networks of heterogeneous infinite-dimensional systems, consisting of components belonging to different system classes, with boundary and in-domain couplings, as long as these systems are well-defined and possess the BIC property.
In particular, our ISS small-gain theorem is applicable to the interconnections of ISS ODE systems, interconnections of evolution equation in Banach spaces and couplings of $n$ time-delay systems. For all these system classes, briefly considered in \linebreak
Sections~\ref{sec:Example: interconnections of ODE systems}, \ref{sec:Example: interconnections of retarded systems}, \ref{sec:Example: interconnections of evolution equations in Banach spaces}, there are natural conditions guaranteeing well-posedness and BIC property \cite{CaH98}, \cite[Chapter 1]{KaJ11b}.
Well-posedness of linear infinite-dimensional systems is treated systematically by methods of admissibility theory, see
\cite{TuW14, JaP04}.
For a systematic analysis of well-posedness of infinite-dimensional systems and for an overview of challenges arising in this way we refer to \cite{TuW14}. For many other classes of systems, in particular, for boundary couplings of heterogeneous PDE systems, well-posedness analysis is a challenging open problem right now.

\subsubsection{Flexibility}
Existing small-gain theorems for delay systems \cite{PTM06, PDT13,TWJ12} and evolution equations in Banach spaces 
\cite{BLJ18} are proved in the maximum formulation. \emph{In this paper, we show the small-gain theorems in sum, max and semimax formulations}, which gives much more flexibility in the choice of right tools for ISS analysis of interconnections. 
Together with the \emph{tightness} of the obtained small-gain theorems, this versatility allows for less restrictive conditions for the stability of networks. Our approach can be also combined with the formalism of so-called \emph{monotone aggregation functions}, see \cite{DRW10, Rue10}, which could further extend the flexibility of obtained results.

\subsubsection{Special cases}

Although we stress upon the significance of small-gain theorems for broad classes of infinite-dimensional systems, our results are novel and powerful already for networks of $n$ time-delay systems or evolution equations in Banach spaces. 
Furthermore, our results hold under minimal requirements on regularity of the systems, which allows to obtain even the classical small-gain theorem for networks of $n$ ODE systems under lesser regularity requirements on the right hand side.

\textbf{ODE systems.} An ISS small-gain theorem for networks of $n$ ODE systems has been shown in \cite[Theorem 11]{DRW07} under assumption that the right-hand side $f$ in \eqref{eq:ODE_coupled_sys} is at least Lipschitz continuous with respect to $x$ on bounded balls (see \cite[p. 97]{DRW07}), or possibly even Lipschitz continuous with respect both $x$ and $u$, as in \cite{DRW07} 
the characterization that ISS $\Iff$ UGS $\wedge$ AG is used for the  proof the ISS small-gain theorem, which was shown in \cite{SoW96} under requirement that $f$ is Lipschitz with respect to both arguments.
%
%

In contrast, in our infinite-dimensional small-gain theorems we require only local existence and uniqueness of solutions (as a part of the definition of a control system) and BIC property for a coupled system. Under this much milder regularity assumption we can invoke the characterization that ISS $\Iff$ UGS $\wedge$ bUAG, and prove the ISS small-gain theorem.

\textbf{Time-delay systems.} 
The interconnections of time-delay systems, briefly introduced in Section~\ref{sec:Example: interconnections of retarded systems} are a special case of control systems, for which our results can be applied. We assume that the subsystems, as well as the coupled system, are well-posed as a control system, and possess BIC property.
Conditions, ensuring this, can be found, e.g., in \cite[Section 7]{MiP20}.

In \cite{TWJ12} ISS small-gain theorems in maximum formulation for couplings of $n\ge 2$ time-delay systems have been obtained by using a Razumikhin-type argument, motivated by \cite{Tee98}. 
In this approach, the delayed state in the right-hand side of a time-delay system is treated as an input to the system, which makes the time-delay system a delay-free system with additional input. To make this approach work, it is assumed in \cite{TWJ12}, that each subsystem in absence of the external inputs is ISS with respect to delayed state input with a gain which is strictly less than 1.
On the one hand, if this assumption holds, then the application of the max-form ISS small-gain theorem \cite[Theorem 2]{TWJ12} becomes simpler than the analysis by means of the max-form ISS small-gain theorem in our paper (Theorem~\ref{thm:ISS_SGT_max_form}).
On the other hand, this additional assumption restricts the class of systems, which can be studied using the approach in \cite{TWJ12}.
In contrast, in our max-form ISS small-gain theorem (Theorem~\ref{thm:ISS_SGT_max_form}) we do not need such an assumption.

In most papers on small-gain theory for delay systems \cite{PTM06, PDT13, TWJ12}, the small-gain theorems are derived in the max-form. Next, we provide an academic example showing that considering only 
max-formulations of the ISS property is too restrictive, and summation and semimaximum ISS small-gain theorems, shown in this paper, can be more natural and less conservative for certain classes of interconnected systems.

\begin{example}
\label{examp:Time-delay-systems-3-systems} 
Consider an interconnection of 3 scalar time-delay systems without external inputs.
\begin{subequations}
\begin{eqnarray}
\dot{x}_1(t) &=& -x_1(t) + x^3_2(t-1) + x^2_3(t-1),\\
\dot{x}_2(t) &=& - x_2(t) + ax_3^{\frac{1}{3}}(t-1),\\
\dot{x}_3(t) &=& - x_3(t) + b|x_1(t-1)|^{\frac{1}{2}}.
\end{eqnarray}
\label{eq:TDS-example-3-subsystems}
\end{subequations}
Here $a,b\in\R$ are some constants. This interconnection as defined in Section~\ref{sec:Example: interconnections of retarded systems}  is well-posed and has BIC property. We use also the notation for the spaces $X_i$ and $X$ from this section.
We denote the state of the $i$-th subsystem at time $t$ with an initial condition $x_i^0$ by
$\phi_i(t,x_i^0)$ and we write $\phi_i^0(t,x_i^0):= \phi_i(t,x_i^0)(0)$.

We derive sufficient conditions for the coefficients $a,b$ guaranteeing the stability of the network based on the ISS small-gain theorem in summation formulation. First we have to verify ISS properties of subsystems. For the first subsystem we obtain:
\begin{eqnarray*}
|\phi_1^0(t,x_1^0)|  &\leq& e^{-t}|x_1^0(0)| + \int_0^t e^{s-t} |x^3_2(s-1) + x^2_3(s-1)| ds \\
                               &\leq& e^{-t}\|x_1^0\|_{X_1} + \|x_2\|^3_{[-1,t]} + \|x_3\|^2_{[-1,t]}.
\end{eqnarray*}
Hence we obtain for the first system the ISS estimate in the form
\begin{eqnarray}
\|\phi_1(t,x_1^0)\|_{X_1} &\leq& e^{-t}\|x_1^0\|_{X_1} + \|x_2\|^3_{[-1,t]} + \|x_3\|^2_{[-1,t]}.
\label{eq:Estimate-1st-subsystem}
\end{eqnarray}
Analogously, one can obtain the estimates for $x_2$-subsystem and $x_3$-subsystem:
\begin{subequations}
\label{eq:Estimates-2-3-subsystems}
\begin{eqnarray}
\|\phi_2(t,x_2^0)\|_{X_2} &\leq& e^{-t}\|x_2^0\|_{X_2} + |a|\|x_1\|^{\frac{1}{3}}_{[-1,t]}, \label{eq:Estimate-2nd-subsystem}\\
\|\phi_3(t,x_3^0)\|_{X_3} &\leq& e^{-t}\|x_3^0\|_{X_3} + |b|\|x_1\|^{\frac{1}{2}}_{[-1,t]}. \label{eq:Estimate-3rd-subsystem}
\end{eqnarray}
\end{subequations}
Thus, all subsystems are ISS and the corresponding gains are: 
\begin{eqnarray}
\gamma_{12}(r) = r^3, \quad \gamma_{13}(r) = r^2, \quad \gamma_{21}(r) = ar^{\frac{1}{3}}, \quad \gamma_{31}(r) = br^{\frac{1}{2}}.
\label{eq:Example-SGC-3-TDS-gains}
\end{eqnarray}
It is easy to see that the conditions 
\begin{eqnarray}
|a|<1\ \ \wedge \ \ |b|<1
\label{eq:Example-SGC-3-TDS-conditions-a-b}
\end{eqnarray}
guarantee that the strong small-gain condition \eqref{eq:Strong_SGC} holds for $\Gamma_\boxplus$ with internal gains \eqref{eq:Example-SGC-3-TDS-gains}, and Theorem~\ref{thm:ISS_SGT} shows ISS of the network under the conditions \eqref{eq:Example-SGC-3-TDS-conditions-a-b}.

Note that ISS small-gain theorem in the maximum formulation (Theorem~\ref{thm:ISS_SGT_max_form} or \cite[Theorem 2]{TWJ12}) could be also applied to the system 
\eqref{eq:TDS-example-3-subsystems}. Using the inequality 
$a+b \leq \max\{(1+\frac{1}{\varepsilon})a,(1+\varepsilon)b\}$, which is valid for all $a,b\geq 0$ and all $\varepsilon>0$,
the inequalities \eqref{eq:Estimate-2nd-subsystem} and \eqref{eq:Estimate-3rd-subsystem} can be reformulated in the maximum form as follows:
\begin{subequations}
\label{eq:Estimates-2-3-subsystems-max}
\begin{eqnarray}
\|\phi_2(t,x_2^0)\|_{X_2} &\leq& \max\big\{(1+\tfrac{1}{\varepsilon})e^{-t}\|x_2^0\|_{X_2}, (1+\varepsilon) |a|\|x_1\|^{\frac{1}{3}}_{[-1,t]}\big\}, \label{eq:Estimate-2nd-subsystem-max}\\
\|\phi_3(t,x_3^0)\|_{X_3} &\leq& \max\big\{(1+\tfrac{1}{\varepsilon})e^{-t}\|x_3^0\|_{X_3}, (1+\varepsilon) |b|\|x_1\|^{\frac{1}{2}}_{[-1,t]}\big\}. \label{eq:Estimate-3rd-subsystem-max}
\end{eqnarray}
\end{subequations}
As $\varepsilon>0$ can be chosen arbitrarily small, this enlargement of gains is not critical for the application of the small-gain theorem, although this is done at the cost of very big enlargement of the transients.
However, max-reformulation of the ISS estimate \eqref{eq:Estimate-1st-subsystem} cannot be performed without a substantial increase of at least one of the internal gains
$\gamma_{12}$ and $\gamma_{13}$, and thus the application of the max-form ISS small-gain theorems (Theorem~\ref{thm:ISS_SGT_max_form} or \cite[Theorem 2]{TWJ12})
leads to more restrictive estimates at the coefficients $a,b$.

Note that if the dynamics of $x_1$-subsystem would have the form 
\[
\dot{x}_1(t) = -x_1(t) + \max\{x^3_2(t-1),x^2_3(t-1)\},
\]
then the most suitable reformulation for the ISS property of subsystems would be the semimaximum formulation. 
\end{example}

\textbf{Weak ISS.} Small-gain theorems in the maximum formulation for weak ISS of coupled time-delay systems have been derived in \cite{PTM06,PDT13}. In this paper we derived weak ISS small-gain theorem in the summation formulation (Theorem~\ref{thm:wISS_SGT}), and semimaximum and maximum formulations can be derived as discussed in Sections~\ref{sec:Semimaximum_formulation}, \ref{sec:Maximum_formulation_new}.\linebreak
This makes it possible to obtain less conservative conditions for weak ISS of interconnections, as discussed in Example~\ref{examp:Time-delay-systems-3-systems}, as well as to treat a larger class of interconnections. 
Note however that in \cite{PTM06,PDT13} also an extension of the weak ISS small-gain theorem to the case of systems with outputs has been shown (weak IOS small-gain theorem), which we do not investigate in this work.

\textbf{Evolution equations with Lipschitz nonlinearities.} 
In \cite{BLJ18} small-gain theorems for couplings of $n$ evolution equations in Banach spaces (see Section~\ref{sec:Example: interconnections of evolution equations in Banach spaces}) 
have been derived, by using a rather different proof technique, which is applicable if the small-gain property is formulated in the maximization formulation. 
Specialized for this class of systems, our ISS small-gain theorem in the maximum formulation is equivalent to \cite[Theorem 3]{BLJ18}, as the cyclic small-gain condition used in \cite[Theorem 3]{BLJ18} is equivalent to the small-gain condition~\eqref{eq:SGC} due to \cite[Theorem 6.4]{Rue10}.
However, as we provide also sum and semimax formulations of ISS small-gain theorem, our results are more flexible. 
At the same time, note that in \cite{BLJ18} the systems with outputs have been considered, and the small-gain theorems for IOS property (which extends ISS for systems with outputs) have been shown, which was not studied in this work. Thus, the developments in \cite{BLJ18} are complementary to this paper.

\section{Characterizations of ISS by means of ULIM with \q{bounded inputs}}
\label{sec:New_ISS_Characterizations}

As bUAG property has been useful in the proof of ISS small-gain theorems, one can expect, that it can be useful also in other contexts. Hence in this section, we characterize ISS in terms of bUAG and bULIM properties, which gives characterizations, which are a bit stronger and more flexible than those, proved in \cite{MiW18b}.


Weaker counterparts of asymptotic gain properties are so-called limit properties.
\begin{definition}
We say that a forward complete system $\Sigma=(X,\Uc,\phi)$ has the
\begin{itemize}
    \item[(i)] \emph{ bounded input uniform limit property (bULIM)}, if there exists
    $\gamma\in\Kinf\cup\{0\}$ so that for every $\eps>0$ and for every $r>0$ there
    exists a $\tau = \tau(\eps,r)$ such that 
for all $x$ with $\|x\|_X \leq r$ and all $u\in\Uc$ there is a $t\leq
\tau$ such that 
\begin{eqnarray}
\|\phi(t,x,u)\|_X \leq \eps + \gamma(\|u\|_{\Uc}).
\label{eq:ULIM_ISS_section}
\end{eqnarray}
  \item[(ii)] \emph{uniform limit property (ULIM)}, if there exists
    $\gamma\in\Kinf\cup\{0\}$ so that for every $\eps>0$ and for every $r>0$ there
    exists a $\tau = \tau(\eps,r)$ such that 
for all $x$ with $\|x\|_X \leq r$ and all $u\in\Uc$ there is a $t\leq
\tau$ such that \eqref{eq:ULIM_ISS_section} holds.
\end{itemize}
\end{definition}
It is easy to see that bUAG implies bULIM and UAG implies ULIM property.

Even nonuniformly globally asymptotically stable forward complete systems do not always have uniform bounds for their reachability sets on finite
intervals (see \cite{MiW18b}). Systems exhibiting such bounds
deserve a special name.
\begin{definition}
\label{Assumption3}
We say that a forward complete system $\Sigma=(X,\Uc,\phi)$ has \emph{bounded reachability sets (BRS)}, if for any $C>0$ and any $\tau>0$ it holds that 
\[
\sup\big\{
\|\phi(t,x,u)\|_X \midset \|x\|_X\leq C,\ \|u\|_{\Uc} \leq C,\ t \in [0,\tau]\big\} < \infty.
\]
\end{definition}
Clearly, UAG property implies bUAG property. It is not completely clear, whether the converse holds without further assumptions, as, e.g., BRS property.
However, for a system satisfying bUAG property, it is not possible in general to verify UAG (and even sAG) property without an increase of the gain, as shown in the next example
\begin{example}
\label{examp:AG_UAG}
Let $X:=\R$ and $\Uc:=L_{\infty}(\R_+,\R)$. Consider a system
        \begin{equation}
        \label{Ex:AG_UAG_Difference}
        \dot{x}=-\tfrac{1}{1+|u(t)|} x.
        \end{equation}
It has a BRS property as $|\phi(t,x,u)|\le |x|$ for all 
$t\geq 0$, $x\in\R$, $u\in\Uc$. 

Let us show that \eqref{Ex:AG_UAG_Difference} has a bUAG property with a zero gain.
Pick any $u\in\Uc$, $x\in\R$ and $t\geq 0$. The corresponding solution of~\eqref{Ex:AG_UAG_Difference} can be estimated as 
\begin{eqnarray*}
|\phi(t,x,u)| &=& e^{-\int_0^t \frac{1}{1+|u(s)|} ds} |x| 
 \leq e^{-\int_0^t \frac{1}{1+\|u\|_{\infty}} ds} |x| 
 = e^{- \frac{1}{1+\|u\|_{\infty}}t} |x|.
\end{eqnarray*}
Now pick any $r>0$ and any $\varepsilon>0$. For any $u\in B_{r,\Uc}$ it holds that 
$|\phi(t,x,u)| \leq  e^{- \frac{1}{1+r}t} r$. 
Now, let $\tau>0$ be so that  $e^{- \frac{1}{1+r}\tau} r = \eps$ (clearly, such $\tau$ exists).

Overall, for all $t\geq \tau$, all $x \in B_r$ and all $u\in B_{r,\Uc}$ it holds that 
$|\phi(t,x,u)| \leq \varepsilon$.
This shows bUAG property of \eqref{Ex:AG_UAG_Difference} with a zero gain.
At the same time, it is known that \eqref{Ex:AG_UAG_Difference} does not have a strong AG property with a zero gain, which is weaker than UAG with a zero gain, see \cite[Remark 4]{Mir16}.
\end{example}

Similarly to Lemma~\ref{lem:UGS_und_bUAG_imply_UAG}, the following result can be shown:  
\begin{lemma}
\label{lem:UGS_und_bULIM_imply_ULIM}
Let $\Sigma=(X,\Uc,\phi)$ be a control system. If $\Sigma$ is UGS and bULIM, then $\Sigma$ is ULIM.
\end{lemma}

Although the notion of bUAG is used not as widely as the standard UAG (but see, e.g., \cite{Tee98}, \cite[Proposition 1.4.3.]{Mir12}), but it is in a certain sense even more natural than UAG. 
Indeed, in characterizations of ISS in terms of UAG and ULIM properties shown in \cite{MiW18b} the UAG (and uniform limit) properties have been applied almost exclusively for uniformly bounded inputs, and as we see in the following theorem, the validity of such characterizations is retained if we use bUAG and bULIM properties instead of UAG and ULIM properties.

\begin{definition}
\label{Assumption2}
Consider a system $\Sigma=(X,\Uc,\phi)$.
We call $0 \in X$ an \emph{equilibrium point} (of the undisturbed system)  if
$\phi(t,0,0) = 0$ for all $t \geq 0$.
\end{definition}

For characterizations of ISS we need one more notion
\begin{definition}
\label{def:RobustEquilibrium_Undisturbed}
Consider a system $\Sigma=(X,\Uc,\phi)$ with equilibrium point $0\in X$.
We say that 
$\phi$ is \emph{continuous at the equilibrium} if   
%
$\forall \eps>0$, $\forall h>0$ \ $\exists \delta = \delta (\eps,h)>0$: 
\begin{eqnarray}
 t\in[0,h]\ \wedge \ \|x\|_X \leq \delta\ \wedge \ \|u\|_{\Uc} \leq \delta \; \Rightarrow \;  \|\phi(t,x,u)\|_X \leq \eps.
\label{eq:RobEqPoint}
\end{eqnarray}
In this case we will also say that $\Sigma$ has the CEP property.
\end{definition}

The following characterizations of ISS refine the main result in \cite{MiW18b}
\begin{theorem}
\label{thm:New_ISS_Characterizations}
Let $\Sigma=(X,\Uc,\phi)$ be a forward complete control system. The following statements are equivalent:
\vspace{-4mm}
   \begin{multicols}{2}
\begin{itemize}
    \item[(i)] $\Sigma$ is ISS.
    \item[(ii)] $\Sigma$ is UAG and UGS.
    \item[(iii)] $\Sigma$ is bUAG and UGS.
    \item[(iv)] $\Sigma$ is bULIM and UGS.
    \item[(v)] $\Sigma$ is bULIM, ULS and BRS.
    \item[(vi)] $\Sigma$ is bUAG, CEP and BRS.
\end{itemize}
 \end{multicols}
\end{theorem}
\vspace{-3mm}
\begin{proof}
The proof goes mostly along the lines of the characterizations of ISS, obtained in \cite[Theorem 5]{MiW18b}.

(vi) $\Rightarrow$ (v). Doing some minimal changes in \cite[Lemma 6]{MiW18b} one can show that bUAG $\wedge$ CEP implies ULS property.

(v) $\Rightarrow$ (iv). Doing some minimal changes in \cite[Proposition 10]{MiW18b} one can show that bULIM $\wedge$ BRS implies uniform global boundedness (UGB), see \cite{MiW18b} for a definition.
By \cite[Lemma 4]{MiW18b} UGB $\wedge$ ULS is equivalent to UGS.

(iv) $\Rightarrow$ (iii). Follows from the proof of \cite[Lemma 7]{MiW18b}.

(iii) $\Rightarrow$ (ii). Follows from Lemma~\ref{lem:UGS_und_bUAG_imply_UAG}.

(ii) $\Rightarrow$ (vi). Clear.

(ii) $\Iff$ (i). Follows by \cite[Lemma 8]{MiW18b}. 
\end{proof}

\begin{remark}[Characterizations and small-gain theorems for strong ISS property]
\label{rem:sISS} 
In \cite{MiW18b} the notions of strong ISS, strong limit property (sLIM) and strong asymptotic gain (sAG) have been defined, which are weaker than ISS, ULIM, and UAG properties respectively. 
It is possible to define the \q{bounded inputs} versions of sLIM and sAG properties, which may be called bsLIM and bsAG, and show that
\begin{center}
strong ISS $\quad\Iff\quad$ bsAG $\wedge$ UGS $\quad\Iff\quad$ bsLIM $\wedge$ UGS.
\end{center}
Since the proof of this result is completely analogous to the proofs of Theorem~\ref{thm:New_ISS_Characterizations} and Lemma~\ref{lem:UGS_und_bUAG_imply_UAG}, we drop definitions and precise formulations of the results.

A harder and more interesting problem is to derive the small-gain theorem for the strong ISS property. The approach, used for the proof of ISS and weak ISS small-gain theorems cannot be straightforwardly adapted for the strong ISS case. Indeed, applying the cocycle property \eqref{eq:Cocycle_property}, which is the starting point in our proof technique, leads to the fact that the term $\bar{\phi}_i\big(t,x_i,(\phi_{\neq i}, u)\big)$ depends on an input $u$. Hence the time $\tau_i$, which is obtained in the next step of the proof of small-gain theorem for ISS property, depends on $u$ as well (since it depends on $\bar{\phi}_i\big(t,x_i,(\phi_{\neq i}, u)\big)$). But this is not what we want in order to achieve sAG (or bsAG) property of the interconnection.
\end{remark}

\vspace{-2mm}

\section*{Acknowledgements}

The author thanks Fabian Wirth for insightful discussions.

\vspace{-7mm}

\bibliographystyle{abbrv}
\bibliography{C:/Users/Andrii/Dropbox/TEX_Data/Mir_LitList_NoMir,C:/Users/Andrii/Dropbox/TEX_Data/MyPublications}



\end{document}